\documentclass[11pt,reqno]{amsart}

\usepackage{marginnote,amssymb,mathrsfs,latexsym,verbatim,pdfsync,color,mathabx}
\usepackage[utf8]{inputenc}

\usepackage[colorlinks=true]{hyperref}
\usepackage{mathtools}

\newcommand{\Ls}{\mathcal{L}}
\newcommand{\Ws}{\mathcal{W}}

\newcommand{\N}{\mathbb{N}}
\newcommand{\Z}{\mathbb{Z}}
\newcommand{\R}{\mathbb{R}}

\newcommand{\loc}{\mathrm{loc}}

\def\e{\mathrm{e}}

\numberwithin{equation}{section}

\DeclareFontFamily{U}{mathx}{\hyphenchar\font45}
\DeclareFontShape{U}{mathx}{m}{n}{
      <5> <6> <7> <8> <9> <10>
      <10.95> <12> <14.4> <17.28> <20.74> <24.88>
      mathx10
      }{}
\DeclareSymbolFont{mathx}{U}{mathx}{m}{n}
\DeclareFontSubstitution{U}{mathx}{m}{n}
\DeclareMathAccent{\widecheck}{0}{mathx}{"71}
\DeclareMathAccent{\wideparen}{0}{mathx}{"75}

\makeatletter
\newcommand{\leqnomode}{\tagsleft@true}
\newcommand{\reqnomode}{\tagsleft@false}
\makeatother

\newcommand{\dd}{\mathrm{d}}

\theoremstyle{theorem}
\newtheorem{theorem}{\sc \textbf{Theorem}}[section]  
\newtheorem{proposition}[theorem]{\sc \textbf{Proposition}}   
\newtheorem{corollary}[theorem]{\sc \textbf{Corollary}}        
\newtheorem{lemma}[theorem]{\sc \textbf{Lemma}}

\theoremstyle{remark}
\newtheorem{definition}[theorem]{\sc Definition}

\title[Besov and Triebel--Lizorkin spaces]{Besov and Triebel--Lizorkin spaces on Lie groups} 
\author[Bruno]{Tommaso Bruno}
\address{Dipartimento di Scienze Matematiche ``Giuseppe Luigi Lagrange'',
  Politecnico di Torino, Corso Duca degli Abruzzi 24, 10129 Torino,
  Italy - Dipartimento di Eccellenza 2018-2022}
\email{tommaso.bruno@polito.it}

\author[Peloso]{Marco M.\ Peloso}
\address{Dipartimento di Matematica, 
Universit\`a degli Studi di Milano, 
Via C.\ Saldini 50,  
20133 Milano, Italy}
\email{marco.peloso@unimi.it}

\author[Vallarino]{Maria Vallarino}
\address{Dipartimento di Scienze Matematiche ``Giuseppe Luigi Lagrange'',
  Politecnico di Torino, Corso Duca degli Abruzzi 24, 10129 Torino,
  Italy - Dipartimento di Eccellenza 2018-2022}
\email{maria.vallarino@polito.it}

\thanks{All authors are  partially supported by the grant PRIN 2015
  {\em Real and Complex Manifolds: Geometry, Topology and Harmonic
    Analysis}, and are members of the Gruppo Nazionale per l'Analisi
  Matematica, la Probabilit\`a e le loro Applicazioni (GNAMPA) of the
  Istituto Nazionale di Alta Matematica (INdAM)}

\begin{document}

\begin{abstract}
In this paper we develop a theory of Besov and  Triebel--Lizorkin spaces on general noncompact Lie groups endowed with a sub-Riemannian structure. Such spaces are defined by means of hypoelliptic sub-Laplacians with drift, and endowed with a measure whose density with respect to a right Haar measure is a continuous positive character of the group. We prove several equivalent characterizations of their norms, we establish comparison results also involving Sobolev spaces of recent introduction, and investigate their complex interpolation and algebra properties. 
\end{abstract}

\maketitle
\begin{center}
{\emph{In memory of Elias M.\ Stein}}
\end{center}

\section{Introduction}
Besov and Triebel--Lizorkin spaces have attracted considerable attention in the last decades, for they encompass several classical function spaces, such as Lebesgue, Sobolev, Hardy and BMO spaces. As such, they have a paramount role in describing the regularity of solutions to differential equations. Following the complete and well-understood theory in the Euclidean setting, see e.g.~\cite{TriebelTFS}, several have been the attempts of generalisation to wider contexts, including Riemannian manifolds with bounded geometry~\cite{Triebelmanifolds, Triebelmanifolds2}, Lie groups endowed with a left-invariant Riemannian structure~\cite{Triebelgroups2, Triebelgroups}, doubling metric measure spaces with a reverse doubling property~\cite{HMY, MY, YZ, KP, HH}, doubling metric measure spaces~\cite{GKN, GKZ}. A theory of Besov spaces has also been developed on Lie groups endowed with a sub-Riemannian structure, first on groups of polynomial growth~\cite{FMV}, see also~\cite{GS}, then recently extended on unimodular groups~\cite{Feneuil}. The aim of the present paper is to develop a satisfactory theory of Besov and  Triebel--Lizorkin spaces on general noncompact Lie groups, potentially nondoubling, endowed with a sub-Riemannian structure. The results we present insert in the theory initiated in~\cite{BPTV}, some of whose results we substantially improve, and are part of a long-term program whose aim is to develop a theory of function spaces on general sub-Riemannian manifolds. 

\smallskip

In the Euclidean setting, the Besov spaces $B^{p,q}_\alpha(\R^d)$ and the Triebel--Lizorkin spaces $F^{p,q}_\alpha(\R^d)$ are classically introduced by means of the Littlewood--Paley decomposition of a function. However, it is well known that if $\Delta$ is the Euclidean nonnegative Laplacian on $\R^d$, then the Besov and Triebel--Lizorkin norms of a distribution $f$ are equivalent respectively to the norms
\begin{equation*}
\|\e^{-t_0 \Delta }f\|_{L^p(\R^d)} + \left( \int_0^1 \left( t^{- \alpha/2} \| (t\Delta)^{m} \e^{-t\Delta} f\|_{L^p(\R^d)}\right)^q \, \frac{\dd t}{t}\right)^{1/q}
\end{equation*}
and
\begin{equation*}
\|\e^{-t_0 \Delta }f\|_{L^p(\R^d)} + \Bigg\| \left( \int_0^1 \left( t^{-\alpha/2} |(t\Delta)^{m} \e^{-t\Delta} f|\right)^q \, \frac{\dd t}{t}\right)^{1/q} \Bigg\|_{L^p(\R^d)},
\end{equation*}
whenever $\alpha\geq 0$, $m>\alpha/2$ is integer and $t_0 \in (0,1)$ (if $\alpha>0$, one can take also $t_0=0$). See, e.g.,~\cite{Triebel2}. By means of these characterizations, which we call of ``Gauss--Weierstrass type'', the problem of defining analogous spaces outside the Euclidean context can be reduced to finding an appropriate substitute for the Laplacian in that context, see e.g.~\cite{KP, HH}. The Littlewood--Paley decomposition, instead, heavily relies on Mihlin--H\"ormander's multiplier theorem for the Laplacian, which is known to fail in some cases, as we explain below.  

If $G$ is a Lie group of polynomial volume growth, endowed with the sub-Riemannian structure induced by a family of left-invariant vector fields satisfying H\"ormander's condition, then a Mihlin--H\"ormander type multiplier theorem holds for the sub-Laplacian associated with the chosen family. In this case, the Besov spaces defined in terms of the Littlewood--Paley decomposition coincide, with equivalence of norms, to those defined by a Gauss--Weierstrass type norm, where the Laplacian is replaced by the sum-of-squares sub-Laplacian associated with the chosen family of vector fields, see~\cite{FMV}. In this case, algebra properties analogous to those in the Euclidean setting hold~\cite{GS}. If $G$ is more generally a unimodular group, then Besov spaces defined by means of a Gauss--Weierstrass type norm were introduced and studied in~\cite{Feneuil}, where it was proved that they still enjoy an algebra property. 

\smallskip

In this paper, we develop a theory of Besov and Triebel--Lizorkin spaces on general noncompact Lie groups endowed with a sub-Riemannian structure. Since these groups might exhibit an exponential volume growth at infinity, in general they do not satisfy a global doubling condition. In particular, our results for Besov spaces extend those in~\cite{GS,Feneuil} while, to the best of our knowledge, those for Triebel--Lizorkin spaces have no counterpart on nondoubling Lie groups endowed with a sub-Riemannian structure. We now precisely describe our setting.

\smallskip

Let $G$ be a noncompact connected Lie group with identity $e$ and let $\mathbf{X}= \{X_1, \dots ,X_\ell \}$ be a family of linearly independent left-invariant vector fields on $G$ satisfying H\"ormander's condition. Denote with $\rho$ a right Haar measure of $G$, with $\delta$ the modular function and let $\chi$ be a continuous positive character of $G$; consider the measure $\mu_\chi$ on $G$ with density $\chi$ with respect to $\rho$, i.e.\ $\dd \mu_\chi = \chi \, \dd \rho$. Consider the differential operator
\begin{equation*}
\Delta_\chi = - \sum_{j=1}^\ell (X_j^2 + c_jX_j), \qquad c_j=(X_j \chi)(e), \quad j=1,\dots, \ell,
\end{equation*}
with domain the set of smooth and compactly supported functions $C_c^\infty(G)$. This operator was introduced by Hebisch, Mauceri and Meda~\cite{HMM}, who showed that $\Delta_\chi$ is essentially self-adjoint on $L^2(\mu_\chi)$. With a slight abuse of notation, we still denote with $\Delta_\chi$ its unique self-adjoint extension. We emphasize that if $\chi$ is the modular function $\delta$, $\mu_\delta= \lambda$ is a left Haar measure of $G$ and the operator $\Delta_\delta$, which from now on will be denoted by $\mathcal L$, is the intrinsic hypoelliptic Laplacian associated with the Carnot--Carath\'eodory metric induced on $G$ by the vector fields ${\bf{X}}$, see~\cite{ABGR,BPTV}, and is the natural substitute of the Laplacian in this setting. This also reflects the fact that the measure $\lambda$ is privileged among the measures $\mu_\chi$. 

In view of the previous discussion, it is worth mentioning that since $\Delta_\chi$ has a holomorphic functional calculus whenever $\chi$ is nontrivial, see~\cite{HMM}, a Mihlin--H\"ormander type theorem for $\Delta_\chi$ and hence an associated Littlewood--Paley decomposition do not hold when $\chi \neq 1$. Thus, given $\alpha \geq 0$, $p,q\in [1,\infty]$, $m>\alpha/2$ and $t_0\in (0,1)$, we are led to define the Besov space $B_\alpha^{p,q}(\mu_\chi)$ and the Triebel--Lizorkin space $F_\alpha^{p,q}(\mu_\chi)$ as the spaces of tempered distributions $f$ such that respectively
\[
 \|\e^{-t_0\Delta_\chi} f\|_{L^p(\mu_\chi)} +   \left( \int_0^1 \left(t^{-\alpha/2}\,  \| (t\Delta_\chi)^{m} \e^{-t\Delta_\chi}  f\|_{L^p(\mu_\chi)}\right)^q \, \frac{\dd t}{t}\right)^{1/q}  
\]
and
\[
\|\e^{-t_0\Delta_\chi} f\|_{L^p(\mu_\chi)} +   \left\| \left( \int_0^1 \left( t^{-\alpha/2} |    (t\Delta_\chi)^{m} \e^{-t\Delta_\chi}  f|\right)^q \, \frac{\dd t}{t}\right)^{1/q} \right\|_{L^p(\mu_\chi)}  
\]
are finite, with the usual modification when $q=\infty$.  As expected, different choices of the parameters $m$ and $t_0$ give equivalent norms (see Theorem~\ref{teo-equiv1}).

\smallskip

We shall obtain several equivalent characterizations of Besov and Triebel--Lizorkin norms. In addition to the aforementioned independence of the parameters $m$ and $t_0$, we realize a discrete version of their norms strongly resembling their definition in the classical case, namely similar in flavour to that obtained by the Littlewood--Paley decomposition of a function (Theorem~\ref{teo-equiv2}). We shall also provide a characterization of the norms in terms of the vector fields of the chosen family ${\bf{X}}$ (Theorem~\ref{teo-equiv3}), a recursive characterization (Theorem~\ref{teo_recursive}), and we also present a characterization in terms of differences of functions (Theorem~\ref{teo_functional}).

Then, we establish various comparison results between Besov and Triebel--Lizorkin spaces, which extend the classical embeddings to the current setting (Theorems~\ref{teo_embeddings},~\ref{teo_embeddings_TL} and~\ref{teo_embeddingp2}). Among these results, we mention that every Triebel--Lizorkin space is intermediate beween two Besov spaces, and that when $q=2$ the Triebel--Lizorkin spaces coincide with the Sobolev spaces introduced in~\cite{BPTV} (see also~\cite{RY}). These results manifest a coherence of the results of~\cite{BPTV} and those of the present paper.

Furthermore, we investigate the complex interpolation properties and the algebra properties of the Besov and Triebel--Lizorkin spaces, see Theorems~\ref{interpolation} and~\ref{teo_algebra} respectively. In particular, we show that if $\alpha>0$, then the spaces $B_\alpha^{p,q}(\mu_\chi)\cap L^{\infty}$ and $F_\alpha^{p,q}(\mu_\chi)\cap L^{\infty}$ are algebras under pointwise multiplication. To prove such algebra properties, we use the paraproduct technique.  

\smallskip

The paper is organized as follows. In Section~\ref{sec:PD}, we introduce the general setting of the paper and give the precise definition of Besov and Triebel--Lizorkin spaces. In Section~\ref{sec:HS}, we establish several results related to the heat semigroup of the operator $\Delta_\chi$. Section~\ref{sec:EN} is devoted to the discussion of various equivalent characterizations of Besov and Triebel--Lizorkin norms.  In Section~\ref{sec:CT} we prove embedding and comparison results, while in Sections~\ref{sec:IP} and~\ref{sec:AP} we obtain interpolation and algebra properties, respectively. In Section~\ref{geo-inq-sec} we discuss further developments of this work. 

\section{Preliminaries and Definitions}\label{sec:PD}
All throughout the paper, $G$ will be a noncompact connected Lie group with identity $e$. We shall denote with $\rho$ a right Haar measure, with $\delta$ the modular function, and with $\lambda$ the left Haar measure such that $\dd \lambda = \delta \, \dd \rho$. We recall that $\delta$ is a smooth positive character of $G$, i.e.\ a smooth homomorphism of $G$ into the multiplicative group $\R^+$.  The letter $\chi$ will always denote a continuous positive character of $G$, which is then automatically smooth. We shall denote with $\mu_\chi$ the measure with density $\chi$ with respect to $\rho$. Observe that $\mu_\delta=\lambda$ and $\mu_1=\rho$.

We fix a family of left-invariant linearly independent vector fields $\mathbf{X}= \{ X_1,\dots, X_\ell\}$ which satisfy H\"ormander's condition. These vector fields induce a left-invariant distance $d_C(\, \cdot\, ,\,  \cdot\, )$ which is the associated Carnot--Carathéodory distance. We let $|x|=d_C(x,e)$, and denote by $B_r$ the ball centred at $e$ of radius $r$. The volume of the ball $B_r$ with respect to the measure $\rho$ will be denoted with $V(r)=\rho(B_r)$; recall that also $V(r)= \lambda(B_r)$. It is well known (cf.~\cite{Guiv,Varopoulos1}) that there exist two constants, which we denote with $d=d(G, \mathbf{X})$ and $D=D(G)$, such that 
\begin{equation}\label{pallepiccole}
C^{-1} r^d \leq V(r) \leq C r^d\qquad \forall r\in (0,1]
\end{equation}
and
\begin{equation}\label{pallegrandi}
V(r)\leq C \e^{Dr}\qquad \forall r\in (1,\infty).
\end{equation}
for a constant $C>0$ independent of $r$. 

If $p\in [1,\infty)$, the spaces of (equivalent classes of) measurable functions whose $p$-power is integrable with respect to $\mu_\chi$ will be denoted by $L^p(\mu_\chi)$, and endowed with the usual norm which we shall denote with $\| \cdot \|_{L^p(\mu_\chi)}$, or $\| \cdot \|_{L^p}$ when there is no risk of confusion. The space $L^\infty$ will be the space of (equivalent classes of) measurable functions which are $\rho$-essentially bounded; observe that this coincides with the space of $\mu_\chi$-essentially bounded functions for every positive character $\chi$ of $G$, since $\mu_\chi$ is absolutely continuous with respect to $\rho$. Observe moreover that, since for every character $\chi$ and $R>0$ there exists a constant $c=c(\chi, R)$ such that
\begin{equation}\label{localdoubmugamma}
c^{-1} \chi(x) \leq \chi (y) \leq c \chi(x) \qquad \forall \, x,y\in G
\, \mbox{ such that }\, d_C(x,y)\leq R,
\end{equation}
cf.~\cite{BPTV}, by~\eqref{pallepiccole} the metric measure space $(G, d_C, \mu_\chi)$ is locally doubling. 

We shall write $\mathscr{I}$ for the set $\{1,\dots, \ell \}$. For every $m\in\N$, $\mathscr{I}^m$ will be the set of multi-indices $J=(j_1,\dots ,j_m)$ such that $j_i \in \mathscr{I}$ and for $J\in \mathscr{I}^m$ we denote by $X_J$ the left-invariant differential operator $X_J = X_{j_1}\cdots X_{j_m}$. Recall that since $\chi$ is a smooth character, if $c_j=(X_j\chi)(e)$ then
\begin{equation}\label{derivatacarattere}
X_j\chi=c_j \chi \qquad \forall j\in \mathscr{I}.
\end{equation}
We denote with $\mathcal{S}(G)$ the Schwartz space of functions $\varphi  \in C^\infty(G)$ such that all the seminorms
\[
\mathcal{N}_{J,n}(\varphi) = \sup_{x\in G} \e^{n|x|} |X_J \varphi(x)|, \qquad n\in \N, \,  J\in \mathscr{I}^m, \; m\in \N,
\]
are finite. We denote by $\mathcal{S}'(G)$ the dual space of $\mathcal{S}(G)$. The convolution between two functions $f$ and $g$, when it exists, is  
	\[
	f*g(x) =\int_G f(xy^{-1})g(y)\, \dd \rho(y).
	\]
Observe that the convolution $f*g$ makes sense also when $f\in \mathcal{S}'(G)$ and $g\in \mathcal{S}(G)$, and that in this case $f*g\in \mathcal{S}'(G)\cap C^\infty(G)$.

\smallskip

Let now $\Delta_{\chi}$ be the sub-Laplacian with drift 
\begin{equation*}
\Delta_\chi = - \sum_{j=1}^\ell (X_j^2 + c_jX_j).
\end{equation*}
We set
\[
X  = \sum\nolimits_{j=1}^\ell c_jX_j, \qquad  \|X\|=\big(\sum\nolimits_{j=1}^{\ell}c_j^2\big)^{1/2}.
\]
The operator $\Delta_\chi$ generates a diffusion semigroup, i.e.\ $(\e^{-t\Delta_\chi})_{t>0}$ extends to a contraction semigroup on $L^p(\mu_\chi)$ for every $p \in [1, \infty]$ (see e.g.~\cite[Proposition 3.1, (ii)]{HMM}) whose infinitesimal generator, with a slight abuse of notation, we still denote with $\Delta_\chi$. Observe that $\Delta_1$ is the standard left-invariant sum-of-squares sub-Laplacian, usually denoted with $\Delta$. The convolution kernel of $\e^{-t\Delta}$ will be denoted with $p_t$, i.e.\ $\e^{-t\Delta} f = f* p_t$. Recall that $p_t\in \mathcal{S}(G)$. Since $\Delta_\chi$ is also left-invariant, $\e^{-t\Delta_\chi}$ admits a convolution kernel as well, which we denote with $p_t^\chi$. Since
\begin{equation}\label{ptgamma_pt}
p_t^\chi= \e^{-\frac{t}{4} \|X\|^2} \chi^{-1/2} p_t,
\end{equation}
cf.~\cite{BPTV}, and since characters grow at most exponentially, cf.~\cite[Proposition 5.7]{HMM}, $p_t^\chi \in \mathcal{S}(G)$. Thus, for $f\in \mathcal{S}'(G)$ one has $\e^{-t\Delta_\chi} f = f * p_t^\chi$.

\smallskip

For every $t>0$ and $m\in\mathbb N$ we denote by $W_t^{(m)}$ the operator
\begin{equation}\label{f: Wtm}
W_t^{(m)}=(t\Delta_\chi)^m\,\e^{-t\Delta_\chi}\,.
\end{equation}
In the case when $\Delta_\chi=\mathcal{L}$, we write
\begin{equation}\label{f: WtmL}
\Ws_t^{(m)}=(t\Ls)^m\,\e^{-t\Ls}\,.
\end{equation} 
As will be clear later on, see e.g.\ Theorem~\ref{teo-equiv2} below, for $j\in\N$ the operators $W^{(m)}_{2^{-j}}$ play an analogous role of the operators $\bigtriangleup_j$ involved in the classical Littlewood--Paley decomposition of a function on $\R^d$ (cf.~\cite{TriebelTFS} for such notation).

\smallskip

We can now define the Besov and Triebel--Lizorkin spaces on $G$ associated with $\Delta_\chi$. If $\beta\geq 0$, we denote with $[\beta]$ the largest integer smaller than or equal to $\beta$.
\begin{definition}\label{maindefinition}
Let $\alpha \geq 0$ and $p,q\in [1,\infty]$. 
\begin{itemize}
\item[(i)] The Besov space $B_\alpha^{p,q}(\mu_\chi)$ is the subspace of $\mathcal{S}'(G)$ made of distributions $f$ such that
\begin{equation}\label{Besov0}
\|f\|_{B_\alpha^{p,q}(\mu_\chi)} \coloneqq \|\e^{-\frac{1}{2}\Delta_\chi} f\|_{L^p(\mu_\chi)} + \mathscr{B}^{p,q}_{\alpha}(f) < +\infty,
\end{equation}
where
\[\mathscr{B}^{p,q}_{\alpha}( f) \coloneqq \left( \int_0^1 \left(t^{-\alpha/2}\,  \| W_t^{([\alpha/2]+1)}   f\|_{L^p(\mu_\chi)}\right)^q \, \frac{\dd t}{t}\right)^{1/q}\]
if $q<\infty$, while 
\[\mathscr{B}^{p,\infty}_{\alpha} (f) \coloneqq \sup_{t\in (0,1)} t^{-\alpha/2} \,\|  W_t^{([\alpha/2]+1)}   f\|_{L^p(\mu_\chi)}.\]
\item[(ii)] The Triebel--Lizorkin space $F_\alpha^{p,q}(\mu_\chi)$ is the subspace of $\mathcal{S}'(G)$ made of distributions $f$ such that
\begin{equation}\label{TL0}
\|f\|_{F_\alpha^{p,q}(\mu_\chi)} \coloneqq \|\e^{-\frac{1}{2}\Delta_\chi} f\|_{L^p(\mu_\chi)} + \mathscr{F}^{p,q}_{\alpha} (f) < +\infty,
\end{equation}
where
\[
\mathscr{F}^{p,q}_{\alpha} (f) \coloneqq \Bigg\| \left( \int_0^1 \left( t^{-\alpha/2} |    W_t^{([\alpha/2]+1)}  f|\right)^q \, \frac{\dd t}{t}\right)^{1/q} \Bigg\|_{L^p(\mu_\chi)}
\]
if $q<\infty$, while 
\[\mathscr{F}^{p,\infty}_{\alpha} (f) \coloneqq \big\| \sup_{t\in (0,1)} t^{-\alpha/2} | W_t^{([\alpha/2]+1)}  f|\big\|_{L^p(\mu_\chi)}.\]
\end{itemize}
\end{definition}
Observe that, for every $p\in [1,\infty)$ and $\alpha\geq 0$, $B^{p,p}_\alpha(\mu_\chi) = F^{p,p}_\alpha(\mu_\chi)$.

\section{The heat semigroup}\label{sec:HS}
In this section we shall prove many results involving the heat semigroup $\e^{-t\Delta_\chi }$ and its associated heat kernel $p_t^\chi$. They will be of fundamental importance later on. For any quantities $A$ and $B$, we shall write $A\lesssim B$ to indicate that there exists a constant $c>0$ such that $A\leq c \,B$. If $A\lesssim B$ and $B\lesssim A$, we write $A\approx B$ . We also set, for any function $g$, $\widecheck{g}(x) = g(x^{-1})$ for every $x\in G$.
\begin{lemma}\label{estimates_ptgamma} 
The following properties hold:
\begin{itemize}
\item[(i)] $(\e^{-t\Delta_\chi})_{t>0}$ is a diffusion semigroup on $(G, \mu_\chi)$;
\smallskip
\item[(ii)] for every $r>0$, $\sup_{B_r}\chi=\e^{\|X\|r}$;
\item[(iii)] there exist two constants $c_1, c_2>0$ such that 
\[(\delta \chi^{-1})^{1/2}(x)\, V(\sqrt t)^{-1}\e^{-c_1|x|^2/t} \lesssim p_t^\chi (x)\lesssim (\delta \chi^{-1})^{1/2}(x) \, V(\sqrt t)^{-1}\e^{-c_2|x|^2/t}
\]
 for every $t\in (0,1)$ and $x\in G$;
\item[(iv)]  for every $h\in \N$ there exists a positive constant $b=b_h$ such that 
\[|X_J p_t^\chi (x)| \lesssim (\delta\chi^{-1})^{1/2}(x)  V(\sqrt t)^{-1}  t^{-\frac h2} \,\e^{-b|x|^2/t}\qquad \forall t\in(0,1),\,x\in G,\, J\in \mathscr{I}^h.
\]
\end{itemize}
\end{lemma}
\begin{proof}
For a proof of~(i) and~(ii), see~\cite[Propositions 3.1 and 5.7]{HMM}. Property~(iii) follows from~\eqref{ptgamma_pt} and~\cite[p.150]{PV}. For property (iv), see~\cite[Lemma 2.3]{BPTV}. 
\end{proof}
From now on, we set $c_3= c_1/c_2$, where $c_1$ and $c_2$ are the constants appearing in Lemma~\ref{estimates_ptgamma}~(iii).
\begin{lemma}\label{pointwiseestheat2}
Let $h, k\in \N$. Then
\begin{itemize}
\item[(i)] for every $0<\kappa' <\kappa$, there exists $C(\kappa/\kappa ')>0$ such that
\[
|\e^{-t \Delta_\chi} g | \lesssim C(\kappa/\kappa ') \e^{-\kappa c_3 t_0\Delta_\chi}|g| \qquad \forall t_0 \in (0,1),\, t\in [\kappa' t_0,\kappa t_0],
\]
\item[(ii)] there exists $a_h>0$ such that
\begin{equation}\label{stimeDeltaX}
|X_J  \e^{-t \Delta_\chi} g |\lesssim t^{-\frac{h}{2}} \e^{-a_h t \Delta_\chi}|g|\qquad \forall t\in (0,1), \, J\in \mathscr{I}^h,
\end{equation}
\item[(iii)] there exists $a_{k,h}>0$ such that  
\begin{equation}\label{stimeXDelta}
|\Delta_\chi^k  \e^{-t \Delta_\chi} X_J g |\lesssim t^{-(k+ \frac h2)} \e^{-a_{k,h} t \Delta_\chi}|g|\qquad \forall t\in (0,1), \, J\in \mathscr{I}^h,
\end{equation}
\end{itemize}
where $g$ is any measurable function in $S'(G)$.
\end{lemma}
\begin{proof}
To prove (i), notice that   
\begin{align*}
|\e^{-t\Delta_\chi} g | = |g* p_{t}^\chi| \leq |g|*p_{t}^\chi. 
\end{align*}
Thus, it is enough to prove that given $t_0\in (0,1)$, $0<\kappa'<\kappa$  
\begin{equation}\label{estheat1}
p_{t}^\chi(x)  \lesssim  p_{\kappa c_3 t_0}^\chi(x)\qquad \forall\, x\in G,\; \forall\, t\in [\kappa' t_0,\kappa t_0],
\end{equation}
which follows by property Lemma~\ref{estimates_ptgamma}~(iii). Indeed
\begin{align*}
p_{t}^\chi(x)
&\lesssim (\delta\chi^{-1})^{1/2}(x) t^{-\frac{d}{2}} \e^{-c_2 \frac{|x|^2}{t}} \lesssim (\delta\chi^{-1})^{1/2}(x) t_0^{-\frac{d}{2}} \e^{-c_2 \frac{|x|^2}{\kappa t_0}} \lesssim p_{\kappa c_3 t_0}^\chi(x)
\end{align*}
which proves~\eqref{estheat1} and concludes the proof of (i).

To prove~(ii), observe that
\begin{align*}
|X_J \e^{-t \Delta_\chi}  g | = |g*  X_J p_{t}^\chi| \leq |g|*| X_Jp_{t}^\chi|,
\end{align*}
and that, by Lemma~\ref{estimates_ptgamma}, there exists $b_{h}>0$ such that
\begin{equation}\label{derdeltapt}
|  X_J p_{t}^\chi(x)| 
\lesssim t^{- \frac{h}{2}} (\delta\chi^{-1})^{1/2}(x)   t^{-\frac{d}{2}} \e^{-b_{h} \frac{|x|^2}{t}}\lesssim  t^{- \frac{h}{2}} p_{a_h t}^\chi(x),
\end{equation}
with $a_h = c_1/b_{h}$.

We now prove~(iii). Observe that
\begin{align*}
\Delta_\chi^k \e^{-t \Delta_\chi}  X_J  g  
&= X_J g* \Delta_\chi^k p_{t}^\chi \\
& =(-1)^{|J|}  g * \{ X_J [(( \Delta_\chi^k p_{t}^\chi) \delta^{-1})^{\vee}]\}^\vee \delta \eqqcolon g * \tilde{p}_t^\chi.
\end{align*}
Since the integral kernel of $\e^{-t\Delta}$ is symmetric, one has $\widecheck{p}_t = \delta^{-1}p_t$, so that
\begin{equation}\label{checkpt}
\widecheck{p}_t^\chi =( \delta^{-1}\chi) p_t^\chi.
\end{equation}
Moreover $(\Delta_\chi^k p_t^\chi)^\vee = (\partial_t^k p_t^\chi)^\vee=  \partial_t^k \widecheck{p}_t^\chi$. Thus,
\[
\tilde{p}_t^\chi =  \delta \{ X_J (( \partial_t^k \widecheck{p}_{t}^\chi) \delta)\}^\vee =  \delta \{ X_J (\chi  \partial_t^k {p}_{t}^\chi)\}^\vee =  \delta\{ X_J (\chi  \Delta_\chi^k {p}_{t}^\chi)\}^\vee.
\]
Now, by~\eqref{derivatacarattere}, 
\[
X_J (\chi  \Delta_\chi^k {p}_{t}^\chi )= \sum_{0\leq |I|\leq h} c_I \chi X_I   \Delta_\chi^k {p}_{t}^\chi 
\]
for suitable coefficients $c_I$. Thus, by \eqref{derdeltapt} and \eqref{checkpt} 
\[
|\delta\{ X_J (\chi  \Delta_\chi^k {p}_{t}^\chi)\}^\vee | \lesssim (\delta\chi^{-1}) \sum_{0\leq |I|\leq h} |(X_I   \Delta_\chi^k {p}_{t}^\chi)^{\vee} | \lesssim (\delta\chi^{-1}) \sum_{0\leq |I|\leq h}  t^{-\frac{|I|}{2}} (\delta^{-1}\chi)    {p}_{a_{2k+h}t}^\chi\,,
\]
which implies~\eqref{stimeXDelta}.
\end{proof}

By Lemma~\ref{pointwiseestheat2}~(i) and~(ii) and by the $L^p$-boundedness of the heat semigroup, we also obtain the following estimates.

 \begin{lemma}\label{pointwiseestheat3}
Let $h\in\mathbb N$ and $p\in [1,\infty]$. For every $g\in L^p(\mu_{\chi})$ 
\begin{itemize}
\item[(i)] $\displaystyle \|X_J \e^{-t\Delta_\chi} g\|_{L^p(\mu_\chi)} \lesssim  t^{-\frac{h}{2}} \|g\|_{L^p(\mu_\chi)}\qquad \forall t\in (0,1), J\in \mathscr{I}^h$,
\item[(ii)]  $\displaystyle \| \e^{-t\Delta_\chi} X_J g\|_{L^p(\mu_\chi)} \lesssim  t^{-\frac{h}{2}} \|g\|_{L^p(\mu_\chi)}\qquad \forall t\in (0,1), J\in \mathscr{I}^h$.
\end{itemize}
\end{lemma}
We now consider an estimate where only the left measure $\lambda$ is involved.

\begin{lemma}\label{normaetdeltaq}
Let $1\leq p\leq p_1\leq \infty$ and $r\geq 1$ be such that $\frac{1}{p} + \frac{1}{r} = 1+ \frac{1}{p_1}$. Then for every $g\in L^p(\lambda)$
\[
\| \e^{-t\Ls} g\|_{L^{p_1}(\lambda)} \lesssim t^{\frac{d}{2}\left( \frac{1}{r}-1\right)} \|g\|_{L^p(\lambda)} \qquad \forall t\in (0,1).
\]
\end{lemma}
\begin{proof}
Arguing as in the proof of~\cite[(20.18)]{HewittRoss}, we have
\[
\|g*p_t^\delta\|_{L^{p_1}(\lambda)} \leq \|g\|_{L^p(\lambda)}\left( \| \widecheck{p}_t^\delta\|_{L^r(\lambda)}^{r/p'} \|p_t^\delta\|_{L^r(\lambda)}^{r/p}\right)
\]
so that it is enough to prove that
\[
\| p^\delta_t\|_{L^r(\lambda)}\lesssim t^{\frac{d}{2}\left( \frac{1}{r}-1\right)},\qquad  \| \widecheck{p}^\delta_t\|_{L^r(\lambda)} \lesssim t^{\frac{d}{2}\left( \frac{1}{r}-1\right)}.
\]
Observe now that by estimates~\eqref{pallepiccole} and~\eqref{pallegrandi}
\begin{equation}\label{rapportovolumipalle}
\frac{V(2^{k+1}\sqrt{t})}{V(\sqrt{t})} \lesssim 2^{dk} \e^{D 2^{k}} \qquad \forall \, k\geq 0, \; t\in (0,1).
\end{equation}
Thus, if for every $k\geq 1$ we denote with $A_{k,t}$ the annulus $B_{2^k\sqrt{t}} \setminus B_{2^{k-1}\sqrt{t}} $,  by Lemma~\ref{estimates_ptgamma}~(ii) and~(iii) there exist positive constants $c$ and $C$ such that
\begin{align*}
\| p_t^\chi \|_{L^r(\lambda)}^r &\lesssim   \int_{B_{\sqrt{t}}}V(\sqrt t)^{-r}\,\dd \lambda + \sum_{k=1}^\infty \e^{-c r 2^{2k}} \e^{C r 2^k\sqrt{t}}  \int_{A_{k,t}}  V(\sqrt t)^{-r} \, \dd \lambda \lesssim t^{-\frac{d}{2}(r-1)}.
\end{align*}
Thanks to~\eqref{checkpt}, the estimate for $\widecheck{p}_t^\chi$ is similar and omitted.
\end{proof}

We shall need several variants of~\cite[Proposition 8]{CRTN}.

\begin{proposition}\label{contCRTN}
Let $p\in(1,\infty)$, $q\in [1, \infty]$ and $0<\kappa'<\kappa$. Then, for every sequence of measurable functions $(t_j)$ such that $t_j(x) \in [\kappa'2^{-j},\kappa 2^{-j}]$ for every $j\in \N$ and $x\in G$, and every sequence $(f_j)$ of measurable functions in $\mathcal S'(G)$,
\[
\Bigg\| \left( \sum_{j=1}^\infty |\e^{-t_j \Delta_\chi} f_j|^q\right)^{1/q}\Bigg\|_{L^p(\mu_\chi)} \lesssim \Bigg\| \left( \sum_{j=1}^\infty |f_j|^q\right)^{1/q}\Bigg\|_{L^p(\mu_\chi)},
\]
with the obvious modification when $q=\infty$, where $\e^{-t_j\Delta_\chi}f$ denotes the function $x\mapsto f*p_{t_j(x)} (x)$.
\end{proposition}
\begin{proof}
The proof is inspired to that of~\cite[Proposition 8]{CRTN}, which covers the case when $t_j$ is constant for every $j$. 

Consider the operator
\[
T_\chi f(x) \coloneqq \sup_{j\in \N} |\e^{-t_j(x)\Delta_\chi}f(x)|,
\]
which is linearizable according to~\cite[Definition 1.20, p.\ 481]{GCRDF}, and bounded on $L^p(\mu_\chi)$, since
\[
T_\chi f(x) \leq \sup_{t>0} |\e^{-t\Delta_\chi}f(x)| \eqqcolon T^*_\chi f(x)
\]
and $T_\chi^*$ is bounded on $L^p(\mu_\chi)$ by the maximal theorem of~\cite[p.\ 73]{Stein}. Moreover, $|T_\chi f|\leq T_\chi |f|$ for $p_t^\chi$ is positive for any $t$. By applying~\cite[Corollary 1.23, p.\ 482]{GCRDF} to $T_\chi$ and observing that $|\e^{-t_j\Delta_\chi} f_j |\leq T_\chi |f_j|$, the conclusion follows when $1<p\leq q$.

We now consider the case when $p>q$. By Lemma~\ref{pointwiseestheat2}, for every function $g\geq 0$
\[
\e^{-t_j(x)\Delta_\chi} g(x) \lesssim \e^{-\kappa c_3 2^{-j}\Delta_\chi} g(x).
\]
We can follow the same argument of~\cite{CRTN}. Indeed, let $r$ be such that $\frac{1}{r}= 1- \frac{q}{p}$ and let $w\geq 0$. Then, since $|\e^{-t_j(x) \Delta_\chi}f(x)|^q \leq \e^{-t_j(x) \Delta_\chi}|f(x)|^q$ by Jensen's inequality, and since $\e^{-t\Delta_\chi}$ is symmetric on $L^2(\mu_\chi)$ for every $t>0$,
\begin{align*}
\left| \int_G |\e^{-t_j(x)\Delta_\chi}f(x)|^qw(x) \, \dd \mu_\chi(x)\right| 
 &\leq \int_G \e^{-t_j(x)\Delta_\chi}|f(x)|^q w(x) \, \dd \mu_\chi(x)\\
& \lesssim \int_G \e^{-\kappa c_3 2^{-j}\Delta_\chi}|f(x)|^q w(x) \, \dd \mu_\chi(x) \\
& = \int_G |f(x)|^q \e^{-\kappa c_3 2^{-j}\Delta_\chi} w(x) \, \dd \mu_\chi(x)\\
& \leq \int_G |f(x)|^q T^*_\chi w(x) \, \dd \mu_\chi(x).
\end{align*}
The proof now can be concluded exactly as in~\cite[pp.\ 307-308]{CRTN} using that $T^*_\chi$ is bounded on $L^p(\mu_\chi)$.
\end{proof}

We now prove a very useful proposition, which is an integral analogue of Proposition~\ref{contCRTN}.

\begin{proposition}\label{CRTNintegrale}
Let $p\in (1,\infty)$, $q\in [1,\infty]$, $c>0$ and $\ell,r \in \R$. Then
\[
\Bigg\| \left( \int_0^1 \left(t^\ell \e^{-ct\Delta_\chi} |G(t,\cdot)|\right)^q \, \frac{\dd t}{t} \right)^{1/q}\Bigg\|_{L^p(\mu_\chi)} \lesssim \Bigg\| \left(  \int_0^1 \left(t^\ell | G(t,\cdot)  |\right)^q \, \frac{\dd t}{t} \right)^{1/q}\Bigg\|_{L^p(\mu_\chi)},
\]
with the usual modification when $q=\infty$, where either $g$ is a measurable function in $\mathcal S'(G)$ and 
\begin{itemize}
\item[(i)] $G(t,x)=\e^{-t\Delta_\chi}g(x)$,
\item[(ii)] $G(t,x)=g(x)$,
\end{itemize}
or $g:(0,1)\times G\rightarrow [0,\infty)$ is such that $g(u,\cdot)$ is a measurable function in $\mathcal S'(G)$ for all $u\in (0,1)$ and 
\begin{itemize}
\item[(iii)] $G(t,x)=  \int_0^1 (u+t)^r g(u,x) \, \dd u  $,
\item[(iv)] $G(t,x)=  \int_t^1  g(u,x) \, \dd u  $.
\end{itemize}
\end{proposition}

\begin{proof}
We first prove (i) when $q<\infty$. Observe that
\begin{align*}
  \int_0^1 \left(t^\ell \e^{-ct\Delta_\chi} |\e^{-t\Delta_\chi}g|\right)^q \, \frac{\dd t}{t} &  \lesssim  \sum_{j=1}^{\infty}\int_{2^{-j}}^{2^{-j+1}} \left( 2^{-j\ell}  \e^{-ct\Delta_\chi}\e^{-(t-2^{-j})\Delta_\chi} |\e^{-2^{-j}\Delta_\chi}g|\right)^q  \, \frac{\dd t}{t}  \\
  &\lesssim  \sum_{j=1}^{\infty} \left(2^{-j\ell}     \e^{-(2c+1)c_3 2^{-j}\Delta_\chi} |\e^{-  2^{-j}\Delta_\chi}g|\right)^q    \,,
\end{align*}
where we have used the fact that $c2^{-j} \leq ct + t- 2^{-j}\leq (2c+1)2^{-j}$ for $t\in [2^{-j}, 2^{-j+1}]$  and we have applied Lemma~\ref{pointwiseestheat2}. Thus, by Proposition~\ref{contCRTN}
\begin{align*}
\Bigg\| \left( \int_0^1 \left(t^\ell \e^{-ct\Delta_\chi} |\e^{-t\Delta_\chi}g|\right)^q \, \frac{\dd t}{t} \right)^{1/q} \Bigg\|_{L^p}
&\lesssim \Bigg\| \left( \sum_{j=1}^{\infty} \left(2^{-j\ell}     \e^{-(2c+1)c_3 2^{-j}\Delta_\chi} |\e^{-  2^{-j}\Delta_\chi}g|\right)^q  \right)^{1/q} \Bigg\|_{L^p}\\
& \lesssim \Bigg\| \left( \sum_{j=1}^\infty \left( 2^{-j\ell} |\e^{-2^{-j}\Delta_\chi} g|\right)^q \right)^{1/q} \Bigg\|_{L^p}\,.
\end{align*}
Observe now that, for every $j\in \N$ and $x\in G$, by the mean value theorem there exists $s_j(x)\in [2^{-j-2}, 2^{-j-1}]$ such that
\[
\int_{2^{-j-2}}^{2^{-j-1}} \left( t^{\ell}  |   \e^{-t\Delta_\chi}g(x)|\right) ^q \, \frac{\dd t}{t} \approx \left(s_j(x)^{\ell}|   \e^{-s_j(x)\Delta_\chi}g(x)|\right)^q.
\]
Then, by applying Lemma~\ref{pointwiseestheat2} and Proposition~\ref{contCRTN} to $t_j(x)=2^{-j} - s_j(x)$, we obtain
\begin{align*}
\Bigg\|\left( \sum_{j=1}^\infty ( 2^{-j\ell} |\e^{-2^{-j}\Delta_\chi} g|)^q \right)^{1/q} \Bigg\|_{L^p}
&=\Bigg\| \left( \sum_{j=1}^\infty ( 2^{-j\ell} |\e^{-(2^{-j}-s_j)\Delta_\chi} \e^{-s_j\Delta_\chi}g|)^q \right)^{1/q} \Bigg\|_{L^p}\\
&  \lesssim\Bigg\|\left(\sum_{j=1}^\infty \left( 2^{-j\ell} | \e^{-s_j\Delta_\chi} g|\right)^q \right)^{1/q} \Bigg\|_{L^p} \\
&  \approx   \Bigg\| \left(\sum_{j=1}^\infty\int_{2^{-j-2}}^{2^{-j-1}} \left( t^{\ell} |  \e^{-t\Delta_\chi} g|\right)^q \, \frac{\dd t}{t} \right)^{1/q}\Bigg\|_{L^p} \\
&  \lesssim\Bigg\| \left(\int_0^1 \left( t^{\ell} |  \e^{-t\Delta_\chi} g|\right)^q \, \frac{\dd t}{t} \right)^{1/q} \Bigg\|_{L^p}.
\end{align*}
We now prove (i) when $q=\infty$. Arguing as above,
\begin{align*}
\sup_{t\in (0,1)} t^\ell \e^{-ct\Delta_\chi} |\e^{-t\Delta_\chi} g| 
&\approx \sup_{j\geq 1} \sup_{t\in [2^{-j}, 2^{-j+1}]} 2^{-j\ell} \e^{-ct\Delta_\chi} |\e^{-t\Delta_\chi} g| \\
&\lesssim \sup_{j\geq 1} 2^{-j\ell}     \e^{-(2c+1)c_3 2^{-j}\Delta_\chi} |\e^{-  2^{-j}\Delta_\chi}g|,
\end{align*}
so that by Proposition~\ref{contCRTN}
\begin{align*}
\Big\| \sup_{t\in (0,1)} t^\ell \e^{-ct\Delta_\chi} |\e^{-t\Delta_\chi} g| \Big\|_{L^p}
&\lesssim \Big\| \sup_{j\geq 1} 2^{-j\ell}     \e^{-(2c+1)c_3 2^{-j}\Delta_\chi} |\e^{-  2^{-j}\Delta_\chi}g|\Big\|_{L^p}\\
&\lesssim \Big\| \sup_{j\geq 1} 2^{-j\ell}    |\e^{-  2^{-j}\Delta_\chi}g|\Big\|_{L^p}\\
&\leq  \Big\| \sup_{j\geq 1}  \sup_{t\in [2^{-j}, 2^{-j+1}]} 2^{-j\ell}    |\e^{-  t\Delta_\chi}g|\Big\|_{L^p}\\
& \lesssim \Big\| \sup_{t\in(0,1)} t^{\ell}    |\e^{-  t\Delta_\chi}g|\Big\|_{L^p}.
\end{align*}
The case (ii) is easier to prove and we omit the details.

To prove (iii) when $q<\infty$, we recall that $g\geq 0$ and use Proposition~\ref{contCRTN}:
\begin{align*}
&\Bigg\| \left( \int_0^1 \left(t^\ell \e^{-c t\Delta_\chi}  \int_0^1 (u+t)^r g(u,x) \, \dd u  \right)^q \, \frac{\dd t}{t} \right)^{1/q}\Bigg\|_{L^p} \\
& \qquad \qquad \lesssim \Bigg\| \left( \sum_{j=1}^\infty \left( 2^{-j\ell } \e^{-c'2^{-j}\Delta_\chi}    \int_0^1 (u+ 2^{-j})^r g(u,x) \, \dd u  \right)^q \right)^{1/q} \Bigg\|_{L^p}\\
&  \qquad \qquad \lesssim \Bigg\|\left( \sum_{j=1}^\infty \left( 2^{-j\ell }   \int_0^1 (u + 2^{-j})^r g(u,x) \, \dd u  \right)^q \right)^{1/q} \Bigg\|_{L^p}\\
& \qquad \qquad  \lesssim \Bigg\| \left( \int_0^1 \left(t^\ell   \int_0^1 (u+t)^r g(u,x) \, \dd u \right)^q \, \frac{\dd t}{t} \right)^{1/q}\Bigg\|_{L^p},
\end{align*}
since  $u+t \approx 2^{-j} +u$ if $t\in [2^{-j}, 2^{-j+1}]$. If $q=\infty$, arguing in the same fashion,
\begin{align*}
\Big\| \sup_{t\in (0,1)} t^\ell \e^{-c t\Delta_\chi}  \int_0^1 (u+t)^r g(u,x) \, \dd u \Big\|_{L^p}
&\lesssim \Big\| \sup_{j\geq 1} 2^{-j\ell} \e^{-c' 2^{-j}\Delta_\chi}  \int_0^1 (u+2^{-j})^r g(u,x) \, \dd u \Big\|_{L^p}\\
&\lesssim\Big\| \sup_{t\in (0,1)} t^{\ell}  \int_0^1 (u+t)^r g(u,x) \, \dd u \Big\|_{L^p}.
\end{align*}
The proof of~(iv) is analogous to~(iii), and we omit it.
\end{proof}

\begin{lemma}\label{bilinearCRTNint}
Let $p\in(1,\infty)$, $q\in [1, \infty]$ and $\ell>0$. Then there exists a positive constant $c$ such that for every measurable functions $f,g$ in $\mathcal S'(G)$ 
\begin{multline*}
\left\| \left( \int_0^1 \left( t^\ell \e^{-t\Delta_\chi} ( |\e^{-t\Delta_\chi}f| \cdot |\e^{-t\Delta_\chi}g|) \right)^q \, \frac{\dd t}{t} \right)^{1/q}\right\|_{L^p(\mu_\chi)} \\
\lesssim \left\| \left(  \int_0^1 \left(t^\ell ( \e^{-ct\Delta_\chi}|\e^{-t\Delta_\chi}f| )\cdot ( \e^{-ct\Delta_\chi} |\e^{-t\Delta_\chi}g|)\right)^q \, \frac{\dd t}{t} \right)^{1/q}\right\|_{L^p(\mu_\chi)},
\end{multline*}
with the usual modification when $q=\infty$.
\end{lemma}
\begin{proof}
We only prove the statement when $q<\infty$. We shall apply repeatedly Lemma~\ref{pointwiseestheat2}. Since $ \e^{-t\Delta_\chi}= \e^{-(t-2^{-j-1})\Delta_\chi}\e^{-2^{-j-1}\Delta_\chi}$, using Proposition~\ref{contCRTN}, 
\begin{align*}
&\left\| \left( \int_0^1 \left( t^\ell \e^{-t\Delta_\chi} \left( |\e^{-t\Delta_\chi}f| |\e^{-t\Delta_\chi}g|\right) \right)^q \, \frac{\dd t}{t} \right)^{1/q}\right\|_{L^p} \\
& \lesssim  \left\| \left(\sum_{j=1}^\infty \int_{2^{-j}}^{2^{-j+1}} \left( 2^{-j\ell } \e^{-2c_3 2^{-j}\Delta_\chi} \left( |\e^{-t\Delta_\chi}f| |\e^{-t\Delta_\chi}g|\right) \right)^q \, \frac{\dd t}{t} \right)^{1/q}\right\|_{L^p}\\
& \lesssim  \left\| \left(\sum_{j=1}^\infty \left( 2^{-j\ell } \e^{-2c_3 2^{-j}\Delta_\chi} \left( \e^{-2c_32^{-j}\Delta_\chi} |\e^{-2^{-j-1}\Delta_\chi}f|  \cdot \e^{-2c_32^{-j}\Delta_\chi} |\e^{-2^{-j-1}\Delta_\chi}g|\right) \right)^q \right)^{1/q}\right\|_{L^p}\\
& \lesssim  \left\| \left(\sum_{j=1}^\infty \left( 2^{-j\ell } \left( \e^{-2c_32^{-j}\Delta_\chi} |\e^{-2^{-j-1}\Delta_\chi}f| \cdot \e^{-2c_32^{-j}\Delta_\chi} |\e^{-2^{-j-1}\Delta_\chi}g|\right) \right)^q \right)^{1/q}\right\|_{L^p}.
\end{align*}
Recall now that for every $t_j \in [2^{-j-3}, 2^{-j-2}]$ one has
\[
|\e^{-2^{-j-1}\Delta_\chi}f| = |\e^{-(2^{-j-1}-t_j)\Delta_\chi} \e^{-t_j\Delta_\chi}f|  \lesssim \e^{-c_32^{-j}\Delta_\chi} |\e^{-t_j\Delta_\chi}f|\,,
\]
hence
\[
\e^{-2c_32^{-j}\Delta_\chi} |\e^{-2^{-j-1}\Delta_\chi}f| \cdot \e^{-2c_32^{-j}\Delta_\chi} |\e^{-2^{-j-1}\Delta_\chi}g|  \lesssim  \e^{-3c_32^{-j}\Delta_\chi} |\e^{-t_j\Delta_\chi}f|\cdot \e^{-3c_32^{-j}\Delta_\chi} | \e^{-t_j\Delta_\chi}g|.
\]
Choose $t_j(x) \in [2^{-j-3}, 2^{-j-2}]$ such that
\begin{align*}
&\left( 2^{-j\ell } \left( \e^{-3c_3 2^{-j}\Delta_\chi} |\e^{-t_j\Delta_\chi}f| \cdot \e^{-3c_32^{-j}\Delta_\chi} |\e^{-t_j\Delta_\chi}g|\right) \right)^q \\
& \qquad = \int_{2^{-j-3}}^{2^{-j-2}} \left( 2^{-j\ell } \left( \e^{-3c_32^{-j}\Delta_\chi} |\e^{-t \Delta_\chi}f| \cdot \e^{-3c_32^{-j}\Delta_\chi} |\e^{-t\Delta_\chi}g|\right) \right)^q \, \frac{\dd t}{t}.
\end{align*}
Then,
\begin{align*}
& \left\| \left(\sum_{j=0}^\infty \left( 2^{-j\ell } \left( \e^{-2c_32^{-j}\Delta_\chi} |\e^{-2^{-j-1}\Delta_\chi}f| \cdot \e^{-2c_32^{-j}\Delta_\chi} |\e^{-2^{-j-1}\Delta_\chi}g|\right) \right)^q \right)^{1/q}\right\|_{L^p}\\
&\qquad  \lesssim \left\| \left(\sum_{j=0}^\infty \left( 2^{-j\ell } \left( \e^{-3c_32^{-j}\Delta_\chi} |\e^{-t_j\Delta_\chi}f| \cdot \e^{-3c_32^{-j}\Delta_\chi} | \e^{-t_j\Delta_\chi}g|\right) \right)^q \right)^{1/q}\right\|_{L^p}\\
& \qquad \lesssim \left\| \left(\sum_{j=0}^\infty  \int_{2^{-j-3}}^{2^{-j-2}} \left( 2^{-j\ell } \left( \e^{-3c_32^{-j}\Delta_\chi} |\e^{-t \Delta_\chi}f| \cdot \e^{-3c_32^{-j}\Delta_\chi} |\e^{-t\Delta_\chi}g|\right) \right)^q \, \frac{\dd t}{t} \right)^{1/q}\right\|_{L^p}\\
& \qquad \lesssim \left\| \left( \int_{0}^{1} \left( t^\ell \left( \e^{-24 c_3^2t \Delta_\chi} |\e^{-t \Delta_\chi}f| \cdot \e^{-24 c_3^2t \Delta_\chi} |\e^{-t\Delta_\chi}g|\right) \right)^q \, \frac{\dd t}{t} \right)^{1/q}\right\|_{L^p},
\end{align*}
which completes the proof.
\end{proof}
As all the discussion above shows, pointwise results concerning the heat semigroup $\e^{-t\Delta_\chi}$ are substantially harder to obtain than $L^p$-norm estimates. This is the reason why, in several of the results presented from now on, we shall give detailed proofs only for those involving Triebel--Lizorkin spaces. In all such cases, the analogous result for Besov spaces can be obtained by means of a similar procedure, but in a somewhat easier fashion.

\section{Equivalent norms}\label{sec:EN}
We begin by stating a fundamental decomposition formula (see~\cite[Lemma 3.1]{Feneuil}) which will be a key ingredient from now on, and which can be thought of, in some sense, as a substitute of the Littlewood--Paley decomposition of a function. If $m\in \N, m\geq 1$ and $f\in \mathcal{S}'(G)$,  then
\begin{equation}\label{LPD}
f= \frac{1}{(m-1)!} \int_0^{1} W_t^{(m)}f \, \frac{\dd t}{t} +  \sum_{k=0}^{m-1} \frac{1}{k!} W_1^{(k)}f,
\end{equation}
where the integral converges in $\mathcal{S}'(G)$. If $f\in \mathcal{S}(G)$, the integral converges in $\mathcal{S}(G)$.

In this section, we prove equivalent characterizations of Besov and Triebel--Lizorkin norms. Observe in particular that for $q\in [1,\infty]$ and $\alpha>0$, they imply the embeddings 
\begin{equation}\label{embedLp}
B_\alpha^{p,q}(\mu_\chi) \hookrightarrow L^p(\mu_\chi), \qquad F_\alpha^{p,q}(\mu_\chi) \hookrightarrow L^p(\mu_\chi)
\end{equation}
for $p\in [1,\infty]$ and $p\in (1,\infty)$ respectively.

\subsection{Independence of parameters}
\begin{theorem} \label{teo-equiv1}
Let $\alpha> 0$, $m>\alpha/2$ be an integer, $t_0\in [0,1)$ and $q\in [1,\infty]$.
\begin{itemize}
\item[(i)] If $p \in [1,\infty]$, then the norm $\|f\|_{B_\alpha^{p,q}(\mu_\chi)} $ is equivalent to the norm
\begin{equation}\label{Besov2}
\left( \int_0^1 \left( t^{-\frac{\alpha}{2}} \| W_t^{(m)} f\|_{L^p(\mu_\chi)}\right)^q \, \frac{\dd t}{t}\right)^{1/q} + \|\e^{-t_0\Delta_\chi}f\|_{L^p(\mu_\chi)},
\end{equation}
with the usual modification when $q=\infty$.
\item[(ii)] If $p\in (1,\infty)$, then the norm $\|f\|_{F_\alpha^{p,q}(\mu_\chi)} $ is equivalent to the norm
\begin{equation}\label{TL1}
\Bigg\| \left( \int_0^1 \left( t^{-\frac{\alpha}{2}}| W_t^{(m)} f|\right)^q \, \frac{\dd t}{t}\right)^{1/q}\Bigg\|_{L^p(\mu_\chi)} + \|\e^{-t_0\Delta_\chi}f\|_{L^p(\mu_\chi)},
\end{equation}
with the usual modification when $q=\infty$.
\end{itemize}
If $\alpha=0$, the norms $\|f\|_{B_\alpha^{p,q}(\mu_\chi)} $ and $\|f\|_{F_\alpha^{p,q}(\mu_\chi)} $ are equivalent to those in~\eqref{Besov2} and~\eqref{TL1}, respectively,  provided $t_0\in (0,1)$.
\end{theorem}

\begin{proof}
We prove only (ii) when $q<\infty$, since the proofs of the remaining cases follow the same steps, and are easier in some respects. We split the proof of (ii) into three steps.

\smallskip

\textit{Step 1}. Let $m > \alpha/2$ be an integer. We prove that if either $t_0\in [0,1)$ and $\alpha>0$, or $t_0\in (0,1)$ and $\alpha=0$, then
\begin{equation}\label{t012}
\|\e^{-t_0 \Delta_\chi}f\|_{L^p} \lesssim  \left\| \left( \int_0^1 \left( t^{-\frac{\alpha}{2}} |W_t^{(m)} f|\right)^q \, \frac{\dd t}{t}\right)^{1/q}\right\|_{L^p} +\|\e^{-\frac{1}{2}\Delta_\chi} f\|_{L^p}
\end{equation}
and
 \begin{equation}\label{12t0}
 \|\e^{-\frac{1}{2}\Delta_\chi}f\|_{L^p} \lesssim \left\| \left( \int_0^1 \left( t^{-\frac{\alpha}{2}} | W_t^{(m)} f|\right)^q \, \frac{\dd t}{t}\right)^{1/q} \right\|_{L^p} +\|\e^{-t_0\Delta_\chi} f\|_{L^p}.
 \end{equation}

Let $\alpha>0$ and $t_0\in [0,1)$. By Lemma~\ref{pointwiseestheat3}, $\|W_1^{(k)} \e^{-t_0 \Delta_\chi} f\|_{L^p} \lesssim \|\e^{-\frac{1}{2}\Delta_\chi}f\|_{L^p}$ for every $k\in \N$. Moreover, $| W_t^{(m)} \e^{-t_0 \Delta_\chi} f| \leq \e^{-t_0 \Delta_\chi} | W_t^{(m)}f|$. We use formula~\eqref{LPD}, these observations, the boundedness of $\e^{-t_0 \Delta_\chi}$ on $L^p(\mu_\chi)$ and H\"older's inequality to obtain
\begin{align*}
\|\e^{-t_0 \Delta_\chi}f\|_{L^p} &\lesssim \left\| \int_0^{1} |W_t^{(m)} \e^{-t_0 \Delta_\chi} f| \, \frac{\dd t}{t} \right\|_{L^p}+  \sum_{k=0}^{m-1} \|W_1^{(k)}\e^{-t_0\Delta_\chi} f\|_{L^p}\\
& \lesssim \left\| \left(\int_0^{1}\left(t^{-\frac{\alpha}{2}} |W_t^{(m)} f|\right)^q \, \frac{\dd t}{t}\right)^{1/q}\right\|_{L^p} +\|\e^{-\frac{1}{2}\Delta_\chi}f\|_{L^p}.
\end{align*}
Let $\alpha = 0$ and $t_0\in (0,1)$. By formula~\eqref{LPD}, Lemma~\ref{pointwiseestheat2}, and arguing as above, we get
\begin{align*}
  \|\e^{-t_0\Delta_\chi} f\|_{L^p} & \lesssim \left\| \int_0^{1} |W_t^{(m+1)} \e^{-t_0 \Delta_\chi} f| \, \frac{\dd t}{t}\right\|_{L^p} + \sum_{k=0}^{m} \|W_1^{(k)} \e^{-t_0 \Delta_\chi} f\|_{L^p} \\
& \lesssim \left\| \e^{- a_2 t_0\Delta_\chi} \int_0^{1} t | W_t^{(m)} f| \, \frac{\dd t}{t}\right\|_{L^p} + \|\e^{-\frac{1}{2}\Delta_\chi}f\|_{L^p} \\
& \lesssim \left\| \left(\int_0^{1} | W_t^{(m)} f|^q \, \frac{\dd t}{t}\right)^{1/q} \right\|_{L^p}+ \|\e^{-\frac{1}{2}\Delta_\chi}f\|_{L^p}.
\end{align*}
This concludes the proof ~\eqref{t012}. The proof of~\eqref{12t0} is analogous and omitted.

\smallskip

\textit{Step 2.} We prove that for all integers $m > \alpha/2$ and $t_0\in (0,1)$,
\[
\Bigg\| \left( \int_0^1 \left( t^{-\frac{\alpha}{2}} | W_t^{(m)} f|\right)^q \, \frac{\dd t}{t}\right)^{1/q} \Bigg\|_{L^p}
\lesssim \Bigg\| \left( \int_0^1 \left( t^{-\frac{\alpha}{2}} | W_t^{(m+1)} f|\right)^q \, \frac{\dd t}{t}\right)^{1/q}\Bigg\|_{L^p} + \|\e^{-t_0\Delta_\chi}f\|_{L^p} .
\]
By~\eqref{LPD} applied to $W_t^{(m)} f$ we get
\begin{align*}
W_t^{(m)} f
& = \int_0^{1} t^{-1} W_t^{(m+1)} \e^{-s\Delta_\chi} f \, \dd s+  W_t^{(m)} \e^{-\Delta_\chi}  f.
\end{align*}
Thus, 
\begin{align*}
 \left\| \left( \int_0^1 \left( t^{-\frac{\alpha}{2}}  |  W_t^{(m)} f|\right)^q \, \frac{\dd t}{t}\right)^{1/q} \right\|_{L^p}
& \leq \Bigg\| \left( \int_0^1 \left( t^{-\frac{\alpha}{2}-1} \int_0^1 | W_t^{(m+1)} \e^{-s\Delta_\chi} f| \, \dd s  \right)^q \, \frac{\dd t}{t}\right)^{1/q}\Bigg\|_{L^p} \\ 
& \quad +\Bigg\|  \left( \int_0^1 \left( t^{-\frac{\alpha}{2}} | W_t^{(m)} \e^{-\Delta_\chi} f|\right)^q \, \frac{\dd t}{t}\right)^{1/q}\Bigg\|_{L^p} \\
&\eqqcolon  I_1 + I_2.
\end{align*}
Now, by Lemma~\ref{pointwiseestheat2} and Proposition~\ref{CRTNintegrale} we obtain that there exists $c>0$ such that
\begin{align}\label{intI2}
I_2 &\lesssim\Bigg\| \left( \int_0^1 \left( t^{m-\frac{\alpha}{2}} \e^{-c \Delta_\chi} | \e^{-t_0\Delta_\chi }f|\right)^q \, \frac{\dd t}{t}\right)^{1/q}\Bigg\|_{L^p} \lesssim \|  \e^{-t_0\Delta_\chi }f \|_{L^p}.
\end{align}
We now consider $I_1$. We split the inner integral according to the splitting $[0,t]\cup [t,1]$.  By Lemma~\ref{pointwiseestheat2} (observe that $ \frac{t}{2}\leq s+\frac{t}{2} \leq 3\frac{t}{2}$ if $s\in [0,t]$)
\begin{align*}
 \int_0^1 \left( t^{-\frac{\alpha}{2}-1} \int_0^t | W_t^{(m+1)} \e^{-s\Delta_\chi} f| \, \dd s  \right)^q \, \frac{\dd t}{t} 
 & \lesssim \int_0^1 \left( t^{-\frac{\alpha}{2}-1} \int_0^t |\e^{-(s+\frac{t}{2})\Delta_\chi}  W_{t/2}^{(m+1)} f|\, \dd s \right)^q \, \frac{\dd t}{t}\\
&\lesssim \int_0^1 \left( t^{-\frac{\alpha}{2}} \e^{-\frac{3}{2}c_3 t\Delta_\chi} |W_{t/2}^{(m+1)}  f| \right)^q \, \frac{\dd t}{t}
\end{align*}
so that by Proposition~\ref{CRTNintegrale} and a change of variables
\begin{equation}\label{int0t}
\Bigg\| \left( \int_0^1 \left( t^{-\frac{\alpha}{2}-1} \int_0^t | W_t^{(m+1)} \e^{-s\Delta_\chi} f| \, \dd s  \right)^q \, \frac{\dd t}{t}\right)^{1/q} \Bigg\|_{L^p} \lesssim\Bigg\| \left( \int_0^1 \left( t^{-\frac{\alpha}{2}} | W_t^{(m+1)} f|\right)^q \, \frac{\dd t}{t}\right)^{1/q}\Bigg\|_{L^p} .
\end{equation}
Moreover, observe that
\begin{align*}
 \int_0^1 \left( t^{-\frac{\alpha}{2}-1} \int_t^1 | W_t^{(m+1)} \e^{-s\Delta_\chi} f| \, \dd s  \right)^q \, \frac{\dd t}{t} 
 & \leq \int_0^1 \left( t^{m-\frac{\alpha}{2}} \e^{-t\Delta_\chi} \int_t^1 s^{-(m+1)} |W_{s}^{(m+1)} f|\, \dd s \right)^q \, \frac{\dd t}{t}\\
\end{align*}
so that
\begin{align*}
&\Bigg\| \left( \int_0^1 \left( t^{-\frac{\alpha}{2}-1} \int_t^1 | W_t^{(m+1)} \e^{-s\Delta_\chi} f| \, \dd s  \right)^q \, \frac{\dd t}{t}\right)^{1/q} \Bigg\|_{L^p} \\ 
& \qquad \lesssim \Bigg \| \left( \int_0^1 \left( t^{m-\frac{\alpha}{2}} \int_t^1 s^{-(m+1)} |W_{s}^{(m+1)} f|\, \dd s \right)^q \, \frac{\dd t}{t} \right)^{1/q}\Bigg\|_{L^p} 
\end{align*}
 by Proposition~\ref{CRTNintegrale}. To conclude, observe that 
\[
\int_0^1 \left( t^{m-\frac{\alpha}{2}} \int_t^1 s^{-(m+1)} |W_{s}^{(m+1)} f|\, \dd s \right)^q \, \frac{\dd t}{t}  = \int_0^1 \left( \int_0^1  K(s,t)g(s)\, \frac{\dd s}{s} \right)^q \, \frac{\dd t}{t}
\]
where $g(s) = s^{-\frac{\alpha}{2}}|W_s^{(m+1)} f |$ and $K(s,t) = \big(\frac{t}{s}\big)^{m-\frac{\alpha}{2}}   \mathbf{1}_{\{s\geq t\}}$.
Since
\[
\sup_{t\in (0,1)} \int_0 ^1K(s,t)\, \frac{\dd s }{s} \lesssim 1 \qquad \text{ and } \qquad \sup_{s\in (0,1)}\int_0 ^1K(s,t)\, \frac{\dd t}{t} \lesssim 1,
\]
Schur's Lemma (see~\cite[Theorem 6.18]{FollandRA}) yields
\[
\Bigg\| \left( \int_0^1 \left( t^{m-\frac{\alpha}{2}} \int_t^1 s^{-(m+1)} |W_{s}^{(m+1)} f|\, \dd s \right)^q \, \frac{\dd t}{t} \right)^{1/q}\Bigg\|_{L^p} \lesssim \Bigg\| \left( \int_0^1 g(t)^q \, \frac{\dd t }{t}\right)^{1/q} \Bigg\|_{L^p},
\]
which together with~\eqref{intI2} and~\eqref{int0t} concludes the proof of Step 2.

\smallskip

\textit{Step 3.}  We prove that for every integer $m> \alpha/2$
\[
\Bigg\|\left( \int_0^1 \left( t^{-\frac{\alpha}{2}} | W^{(m+1)}_t  f|\right)^q \, \frac{\dd t}{t}\right)^{1/q}\Bigg\|_{L^p} \lesssim\Bigg\| \left( \int_0^1 \left( t^{-\frac{\alpha}{2}} | W^{(m)}_t f|_{L^p}\right)^q \, \frac{\dd t}{t}\right)^{1/q}\Bigg\|_{L^p}.
\]
By Lemma~\ref{pointwiseestheat2} and Proposition~\ref{CRTNintegrale}, there exists $c>0$ such that
\begin{align*}
  \left\| \left(\int_0^1 \left( t^{-\frac{\alpha}{2}} |W^{(m+1)}_t f|\right)^q \, \frac{\dd t}{t} \right)^{1/q}\right\|_{L^p}
& \lesssim \left\| \left( \int_0^1 \left( t^{-\frac{\alpha}{2}} \e^{-ct \Delta_\chi} |W^{(m)}_{t/2} f|\right)^q \, \frac{\dd t}{t}\right)^{1/q}\right\|_{L^p}\\
& \lesssim \left\| \left( \int_0^1 \left( t^{-\frac{\alpha}{2}} | W^{(m)}_{t/2} f|\right)^q \, \frac{\dd t}{t}\right)^{1/q}\right\|_{L^p},
\end{align*}
which concludes the proof of Step 3 and of the equivalence of~\eqref{TL0} and~\eqref{TL1}.
\end{proof}

\subsection{Littlewood--Paley type characterization}
The next result concerns a characterization that resembles the definition of Besov and Triebel--Lizorkin norms in the classical case, which makes use of the Littlewood--Paley decomposition of a function. As already mentioned, for $j\in\N$ the operators $W^{(m)}_{2^{-j}}$ play the role of the operators $\bigtriangleup_j$ in the classical Littlewood--Paley decomposition, while $\e^{-t_0\Delta_\chi}$ plays the role of $S_0$.
\begin{theorem} \label{teo-equiv2}
Let $\alpha>  0$, $m>\alpha/2$ be an integer, $t_0\in [0,1)$ and $q\in [1,\infty]$.
\begin{itemize}
\item[(i)] If $p\in [1,\infty]$, then the norm $\|f\|_{B_\alpha^{p,q}(\mu_\chi)} $ is equivalent to the  norm
\begin{equation}\label{Besov3}
\left(\sum_{j=0}^\infty \left( 2^{j \frac{\alpha}{2}} \|W^{(m)}_{2^{-j}} f\|_{L^p(\mu_\chi)} \right)^q \right)^{1/q} + \|\e^{-t_0\Delta_\chi}f\|_{L^p(\mu_\chi)},
\end{equation}
with the usual modification when $q=\infty$.
\item[(ii)]  If $p\in (1, \infty)$, then the norm $\|f\|_{F_\alpha^{p,q}(\mu_\chi)} $ is equivalent to the norm
\begin{equation}\label{TL2}
\Bigg\| \left(\sum_{j=0}^\infty \left( 2^{j \frac{\alpha}{2}} |W^{(m)}_{2^{-j}} f|\right)^q \right)^{1/q} \Bigg\|_{L^p(\mu_\chi)} + \| \e^{-t_0\Delta_\chi}f\|_{L^p(\mu_\chi)},
\end{equation}
with the usual modification when $q=\infty$.
\end{itemize}
If $\alpha=0$, the norms $\|f\|_{B_\alpha^{p,q}(\mu_\chi)} $ and $\|f\|_{F_\alpha^{p,q}(\mu_\chi)} $ are equivalent respectively to those in~\eqref{Besov3} and~\eqref{TL2} provided $t_0\in (0,1)$.
\end{theorem}

\begin{proof}
Again, we prove the statements only when $q<\infty$. To prove (i), just observe that
\begin{align*}
 \sum_{j=0}^\infty \left( 2^{j\frac{\alpha}{2}} \|W^{(m)}_{2^{-j}} f\|_{L^p} \right)^q
& \approx \sum_{j=0}^\infty \int_{2^{-j-1}}^{2^{-j}}  \left( 2^{j \frac{\alpha}{2}} \| \e^{-(2^{-j}-t)\Delta_\chi}W^{(m)}_t f\|_{L^p}  \right)^q \, \frac{\dd t}{t}  \\
& \lesssim \int_{0}^{1}  \left( t^{-\frac{\alpha}{2}} \|W^{(m)}_t  f\|_{L^p}  \right)^q \, \frac{\dd t}{t}  \\
&\lesssim \sum_{j=0}^\infty \left( 2^{j \frac{\alpha}{2}} \|W^{(m)}_{2^{-j}}  f\|_{L^p} \right)^q.
\end{align*}
Thus the norm~\eqref{Besov3} is equivalent to the norm~\eqref{Besov2}, and the conclusion follows by Theorem~\ref{teo-equiv1}.

To prove (ii), we shall prove that
\begin{equation}\label{upsilondiscreto}
\Bigg\| \left( \int_0^1 \left( t^{-\frac{\alpha}{2}}| W_t^{(m)} f|\right)^q \, \frac{\dd t}{t}\right)^{1/q}\Bigg\|_{L^p} \approx \Bigg\| \left(\sum_{j=0}^\infty \left( 2^{j\frac{\alpha}{2}} |W^{(m)}_{2^{-j}} f|\right)^q \right)^{1/q} \Bigg\|_{L^p},
\end{equation}
which yields the conclusion by Theorem~\ref{teo-equiv1}. Its proof is similar to that of Proposition~\ref{CRTNintegrale}.

To prove the inequality $\lesssim$, we observe that
\begin{align*}
 \int_0^1 \left( t^{-\frac{\alpha}{2}} |W_t^{(m)} f|\right)^q \, \frac{\dd t}{t} 
 & \approx \sum_{j=1}^\infty  \int_{2^{-j}}^{2^{-j+1}} \left( 2^{j \frac{\alpha}{2}} |\e^{-(t-2^{-j-1})\Delta_\chi} W^{(m)}_{2^{-j-1}}  f|\right)^q \, \frac{\dd t}{t}\\
& \lesssim \sum_{j=1}^\infty \left( 2^{j \frac{\alpha}{2}} \e^{-{4c_3}2^{-j}\Delta_\chi}|W^{(m)}_{2^{-j-1}}  f|\right)^q,
\end{align*}
since $2^{-j-1}\leq t-2^{-j-1} \leq 2^{j+2}$ for every $t\in [2^{-j}, 2^{-j+1}]$, so that Lemma~\ref{pointwiseestheat2} applies. Thus, by Proposition~\ref{contCRTN},
\begin{align}\label{TL2leq}
\Bigg\|\left( \int_0^1 \left( t^{-\frac{\alpha}{2}}| W^{(m)}_t  f|\right)^q \, \frac{\dd t}{t} \right)^{1/q}\Bigg\|_{L^p}
 &\lesssim  \Bigg\|\left( \sum_{j=0}^\infty  \left( 2^{j \frac{\alpha}{2}} | W^{(m)}_{2^{-j}}  f|\right)^q\right)^{1/q}\Bigg\|_{L^p}.
\end{align}
This proves the inequality $\lesssim$ of the statement.

We now prove the inequality $\gtrsim$. By the mean value theorem, there exists  $t_j(x)\in [2^{-j-2}, 2^{-j-1}]$ such that
\[
\int_{2^{-j-2}}^{2^{-j-1}} \left( t^{-\frac{\alpha}{2}}| W^{(m)}_t f(x)|\right) ^q \, \frac{\dd t}{t} \approx \left(t_j(x)^{-\frac{\alpha}{2}}| W^{(m)}_{t_j(x)} f(x)|\right)^q,
\]
where $W^{(m)}_{t_j(x)} f(x)$ is the function $x\mapsto t_j(x)^m [(\Delta_\chi^m f) * p_{t_j(x)}^\chi](x)$. Then, by Proposition~\ref{contCRTN}
\begin{align*}
\Bigg\| \left(\sum_{j=0}^\infty \left( 2^{j\frac{\alpha}{2}}|W^{(m)}_{2^{-j}} f|\right)^q \right)^{1/q} \Bigg\|_{L^p} 
 & = \Bigg\| \left(\sum_{j=0}^\infty \left( 2^{j \frac{\alpha}{2}} | \e^{-(2^{-j}- t_j)\Delta_\chi} W^{(m)}_{t_j} f|\right)^q \right)^{1/q} \Bigg\|_{L^p} \\
&  \lesssim\Bigg\| \left(\sum_{j=0}^\infty \left( 2^{j \frac{\alpha}{2}} | W^{(m)}_{t_j} f|\right)^q \right)^{1/q} \Bigg\|_{L^p} \\
&  \approx  \Bigg\| \left(\sum_{j=0}^\infty\int_{2^{-j-2}}^{2^{-j-1}} \left( t^{-\frac{\alpha}{2}} |W^{(m)}_t f|\right)^q \, \frac{\dd t}{t} \right)^{1/q} \Bigg\|_{L^p} \\
&  \lesssim \Bigg\| \left(\int_0^1 \left( t^{-\frac{\alpha}{2}} | W^{(m)}_t  f|\right)^q \, \frac{\dd t}{t} \right)^{1/q} \Bigg\|_{L^p},
\end{align*}
which completes the proof.
\end{proof}

To proceed further, we shall need the following lemma. It consists of two statements which are both corollaries of Schur's Lemma.
\begin{lemma} \label{lemmasomme}
Let $0<\gamma<\eta$ be two real numbers and $q\in [1,\infty]$. 
\begin{itemize}
\item[(i)] If $a,b\in \Z\cup\{\pm \infty\}$ are such that $a<b$, then for any sequence $(d_n)_{n\in \Z} \subset [0,\infty)$
\[
\sum_{j=a}^b \left( 2^{-j\gamma} \sum_{n=a}^b 2^{\min \{n,j\}\eta} d_n \right)^q \lesssim \sum_{n=a}^b \left(2^{(-\gamma+\eta)n} d_n\right)^q.
\]
\item[(ii)] For every function $d \colon (0,1) \to [0,\infty)$
\[
\int_{0}^1 \left( u^\gamma \int_0^1  \frac{1}{(t+u)^{\eta}} d(t)\, \frac{\dd t}{t} \right)^q \, \frac{\dd u}{u} \lesssim  \int_{0}^1 \left( t^{\gamma - \eta} d(t)\right)^q \, \frac{\dd t}{t}.
\]
\end{itemize}
The obvious modification applies when $q=\infty$.
\end{lemma}

\begin{proof} The statement (i) is~\cite[Lemma 2.2]{Feneuil}. To prove (ii), we rewrite the integral as
\[
\int_{0}^1 \left( u^\gamma \int_0^1  \frac{1}{(t+u)^{\eta}} d(t)\, \frac{\dd t}{t} \right)^q \, \frac{\dd u}{u} = \int_{0}^1 \left( \int_0^1 K(t,u) \tilde d(t) \, \frac{\dd t}{t} \right)^q \, \frac{\dd u}{u} 
\]
where $\tilde d(t) = t^{\gamma-\eta} d(t)$ and $K(t,u) =( \frac{u}{t})^\gamma  \frac{t^\eta}{(t+u)^{\eta}}$. Since
\[
\sup_{t\in (0,1)} \int_0^1 K(t,u)\, \frac{\dd u}{u} \lesssim 1, \qquad\qquad  \sup_{u\in (0,1)} \int_0^1 K(t,u)\,\frac{\dd t}{t} \lesssim 1,
\]
the statement follows by Schur's Lemma.
\end{proof}

\subsection{Characterizations in terms of vector fields}
The following result is a characterization of Besov and Triebel--Lizorkin norms in terms of the vector fields in $\mathbf{X}$. We write
\[
W^{(m), *}_{2^{-j}} f = \sup_{t\in [2^{-j},2^{-j+1}]}\max_{|J|\leq 2m} | t^m X_J \e^{-t\Delta_\chi} f|  \approx 2^{-jm} \sup_{t\in [2^{-j},2^{-j+1}]}\max_{|J|\leq 2m} |X_J \e^{-t\Delta_\chi} f| . 
\]
\begin{theorem}\label{teo-equiv3}
Let $\alpha > 0$, $m >\alpha/2$ be an integer and $q\in [1,\infty]$. 
\begin{itemize}
\item[(i)] If $p\in [1,\infty]$, then the norm $\|f\|_{B_\alpha^{p,q}(\mu_\chi)}$ is equivalent to the norm
\begin{equation}\label{Bmaxsup} 
\left(\sum_{j=0}^\infty \left(2^{j \frac{\alpha}{2}} \| W^{(m), *}_{2^{-j}} f \|_{L^p(\mu_\chi)} \right)^q \right)^{1/q} + \|f\|_{L^p(\mu_\chi)},
\end{equation}
with the usual modification when $q=\infty$.
\item[(ii)] If  $p\in (1,\infty)$, then the norm $\|f\|_{F_\alpha^{p,q}(\mu_\chi)}$ is equivalent to the norm
\begin{equation}\label{TLmaxsup} 
\Bigg\|\left(\sum_{j=0}^\infty \left(2^{j\frac{\alpha}{2}}W^{(m), *}_{2^{-j}} f  \right)^q \right)^{1/q}\Bigg\|_{L^p(\mu_\chi)} + \|f\|_{L^p(\mu_\chi)},
\end{equation}
with the usual modification when $q=\infty$.
\end{itemize}
\end{theorem}
\begin{proof}
We prove only (ii), for the proof of (i) is similar and easier in some respects. Since
\[ |\Delta_\chi^{m} \e^{-2^{-j}\Delta_\chi}f | \lesssim \sup_{t\in [2^{-j},2^{-j+1}]} \max_{|J| \leq 2m} | X_J \e^{-t\Delta_\chi}f|,
\]
the inequality $\|f\|_{F^{p,q}_{\alpha}(\mu_\chi)} \lesssim\eqref{TLmaxsup} $ is immediate.

We now prove the converse inequality. By~\eqref{LPD}, we may write
\begin{align*}
f &= \frac{1}{(m-1)!} \int_0^{1} W^{(m)}_s f \, \frac{\dd s}{s} + \sum_{h=0}^{m-1}\frac{1}{h!} W^{(h)}_1 f = \frac{1}{(m-1)!}\sum_{n=1}^\infty f_n +  \sum_{h=0}^{m-1}\frac{1}{h!} W^{(h)}_1  f,
\end{align*}
where
\[
 f_n = \int_{2^{-n}}^{2^{-n+1}} W^{(m)}_s  f \, \frac{\dd s}{s}.
\]
Since by Proposition \ref{contCRTN} there exists $c>0$ such that 
\[
|X_J \e^{-t\Delta_\chi} W^{(h)}_1 f| \lesssim \e^{-c\Delta_\chi} |f|
\]
for every $J$ such that $|J|\leq 2m$, $t\in [2^{-j}, 2^{-j+1}]$ and $h\in \{0,\dots, m-1\}$, when $q<\infty$ we obtain
\[
\Bigg\| \left(\sum_{j=0}^\infty \left(2^{j \frac{\alpha}{2}} |W^{(m), *}_{2^{-j}} W^{(h)}_1 f| \right)^q \right)^{1/q}\Bigg\|_{L^p}  \lesssim \| f\|_{L^p}.
\]
Observe now that
\begin{align*}
X_J \e^{-t \Delta_\chi} f_n  
&= X_J \e^{-(2^{-n-1} + t)\Delta_\chi} \int_{2^{-n}}^{3\, 2^{-n-1}} (s\Delta_\chi)^m \e^{-(s-2^{-n-1})\Delta_\chi} f \, \frac{\dd s}{s}  \\ & \hspace{1cm}+ X_J \e^{-(2^{-n} +t)\Delta_\chi} \int_{3\,2^{-n-1}}^{2^{-n+1}} (s\Delta_\chi)^m \e^{-(s-2^{-n})\Delta_\chi} f \, \frac{\dd s}{s} \\
 &= X_J \e^{-(2^{-n-1}+t)\Delta_\chi} \int_{2^{-n-1}}^{2^{-n}} (s+2^{-n-1})^m\Delta_\chi^m \e^{-s\Delta_\chi} f \, \frac{\dd s}{s+2^{-n-1}} \\ &\hspace{1cm} + X_J \e^{-(2^{-n}+t)\Delta_\chi} \int_{2^{-n-1}}^{2^{-n}} (s+2^{-n})^m \Delta_\chi^m \e^{-s\Delta_\chi} f \, \frac{\dd s}{s+2^{-n}}. 
\end{align*}
If $|J|\leq 2m$ and $t\in [2^{-j}, 2^{-j+1}]$, by Lemma~\ref{pointwiseestheat2} we have
\begin{align*}
 |X_J \e^{-t \Delta_\chi}f_n |
& \lesssim [2^{-n}+2^{-j}]^{-m} \e^{-c(2^{-j}+2^{-n})\Delta_\chi} \int_{2^{-n-1}}^{2^{-n}} |W^{(m)}_s f | \, \frac{\dd s}{s}
 \end{align*}
 for some $c>0$. In other words, once we define 
\[g_n =  \e^{-c2^{-n}\Delta_\chi} \int_{2^{-n-1}}^{2^{-n}} |W^{(m)}_s f | \, \frac{\dd s}{s},\]
we have
\begin{equation*} \label{estdeltam}
\sup_{t\in [2^{-j},2^{-j+1}]}\max_{|J|\leq 2m}  |X_J \e^{-t \Delta_\chi}f_n | \lesssim 2^{m\min \{j,n\}} \e^{-c 2^{-j}\Delta_\chi} g_n.
\end{equation*}
Therefore, by Proposition~\ref{contCRTN} and Lemma~\ref{lemmasomme}~(i), when $q<\infty$
\begin{align*}
\Bigg\|  \left(\sum_{j=0}^\infty \left(2^{j \frac{\alpha}{2}} | W^{(m), *}_{2^{-j}} f  | \right)^q \right)^{1/q}\Bigg\|_{L^p}
  &\lesssim \Bigg\| \sum_{j=0}^\infty  \left(2^{-j(m-\frac{\alpha}{2})} \e^{-c 2^{-j}\Delta_\chi}  \sum_{n=1}^\infty 2^{m\min \{j,n\}} g_n \right)^q\Bigg\|_{L^p} \\
&\lesssim \Bigg\| \sum_{n=1}^{\infty} \left(2^{n\frac{\alpha}{2}} g_n \right)^q\Bigg\|_{L^p}.
\end{align*}
Since by Lemma~\ref{pointwiseestheat2} there exists $c'>0$ such that
\[
g_n  \lesssim \e^{-c' 2^{-n}\Delta_\chi }|W^{(m)}_{2^{-n-2}} f|,
\]
Proposition~\ref{contCRTN} completes the proof of (ii) when $q<\infty$. We leave the details of the case $q=\infty$ to the reader.
\end{proof}

\subsection{Recursive characterizations}
As Sobolev spaces (see~\cite[Proposition 3.4]{BPTV}), also Besov and Triebel--Lizorkin spaces can be characterized recursively.
\begin{theorem}\label{teo_recursive}
Let $\alpha>0$ and $q\in [1,\infty]$.
\begin{itemize}
\item[(i)] If $p\in [1,\infty]$, then $f\in B^{p,q}_{\alpha+1}(\mu_\chi)$ if and only if $f\in L^p(\mu_\chi)$ and  $X_j f \in B^{p,q}_\alpha(\mu_\chi)$ for every $j\in \mathscr{I}$. In particular
\[
\|f\|_{B^{p,q}_{\alpha+1}(\mu_\chi)}\approx  \sum_{j=1}^\ell \|X_j f\|_{B^{p,q}_\alpha(\mu_\chi)} + \|f\|_{L^p(\mu_\chi)}.
\]
\item[(ii)] If $p\in (1,\infty)$, then $f\in F^{p,q}_{\alpha+1}(\mu_\chi)$ if and only if $f\in L^p(\mu_\chi)$ and  $X_j f \in F^{p,q}_\alpha(\mu_\chi)$ for every $j\in \mathscr{I}$. In particular
\[
\|f\|_{F^{p,q}_{\alpha+1}(\mu_\chi)}\approx \sum_{j=1}^\ell \|X_j f\|_{F^{p,q}_\alpha(\mu_\chi)} +\|f\|_{L^p(\mu_\chi)}.
\] 
\end{itemize}
\end{theorem}

\begin{proof}
We prove (ii), for the proof of (i) follows the same steps and is easier. We claim that for every $p\in (1,\infty)$, $q\in [1,\infty]$, $\beta > -1$ and $i\in  \{ 1,\dots, \ell\}$
\begin{equation}\label{claimalphaTL}
\mathscr{F}^{p,q}_{\beta}(X_i f)  \lesssim \mathscr{F}^{p,q}_{\beta + 1}( f) + \|f\|_{L^p} \approx \|f\|_{F^{p,q}_{\beta+1}},
\end{equation}
where we extended the definition of $\mathscr{F}^{p,q}_{\beta}$ to the case when $-1<\beta \leq 0$, by putting $[\beta]=0$ in that case.  Assuming the claim, we prove the theorem. Indeed, if $\alpha>0$, by the claim with $\beta=\alpha-1$
\begin{align*}
\mathscr{F}^{p,q}_{\alpha+1} (f)
 & = \mathscr{F}^{p,q}_{\alpha-1}  (\Delta_\chi f)\\
 &  \lesssim \sum_{i=1}^\ell \mathscr{F}^{p,q}_{\alpha-1} (X_i(X_if)) + \sum_{i=1}^\ell \mathscr{F}^{p,q}_{\alpha-1} (X_if)\\
 & \lesssim \sum_{i=1}^\ell ( \|X_i f\|_{F^{p,q}_\alpha} +  \|X_i f\|_{F^{p,q}_{\alpha-1}} ) \\
 & \lesssim \sum_{i=1}^\ell \|X_i f\|_{F^{p,q}_\alpha},
\end{align*}
which proves the inequality $\lesssim$ of the statement. The converse inequality also follows, since $\mathscr{F}^{p,q}_{\alpha} (X_i f)\lesssim \|f\|_{F^{p,q}_{\alpha+1}}$ by the claim with $\beta=\alpha$ and
\[
\|\e^{-\frac{1}{2}\Delta_\chi} X_i f\|_{L^p} \lesssim \|f\|_{L^p} \leq \|f\|_{F^{p,q}_\alpha}
\]
by Lemma~\ref{pointwiseestheat3}. Thus, it remains to prove the claim~\eqref{claimalphaTL}.

\smallskip

Let $\overline{m} = [(\beta+1)/2 ]+ 1$ and $m = [\beta/2 ]+ 1$.  By~\eqref{LPD}
\[
f = \frac{1}{(\overline{m}-1)!} \int_0^{1} W^{(\overline{m})}_t f \, \frac{\dd t}{t} +  \sum_{k=0}^{\overline{m}-1} \frac{1}{k!} W^{(k)}_1 f = \frac{1}{(\overline{m}-1)!}\sum_{n=1}^{\infty} f_n +  \sum_{k=0}^{\overline{m}-1} \frac{1}{k!} W^{(k)}_1 f,
\] 
where
\[f_n = \int_{2^{-n}}^{2^{-n+1}} W^{(\overline{m})}_t f \, \frac{\dd t}{t}.
\]
Hence
\[
\mathscr{F}_{\beta}^{p,q}(X_i f) \lesssim \mathscr{F}_{\beta}^{p,q}\left( X_i\sum_{n=1}^{\infty} f_n\right) +\mathscr{F}_{\beta}^{p,q}\left( \sum_{k=0}^{\overline{m}-1} X_i W^{(k)}_1  f\right).
\]
Define $g_n$ by
 \begin{align*}
   f_n & = \e^{-2^{-n-2}\Delta_\chi} \, 2^{\overline{m}} \int_{2^{-n}}^{2^{-n+1}} \e^{-(\frac{t}{2}-2^{-n-2})\Delta_\chi} W^{(\overline{m})}_{t/2} f\, \frac{\dd t}{t}  \eqqcolon \e^{-2^{-n-2}\Delta_\chi}  g_n.
\end{align*}
Notice that by Lemma~\ref{pointwiseestheat2} there exists $a_{m,1}>0$ such that
\begin{equation}\label{eqpermaxsup1}
\begin{aligned}
|\Delta_\chi^{m} \e^{-2^{-j}\Delta_\chi} X_i f_n |  
& \lesssim  2^{j(m+\frac{1}{2})} \e^{-a_{m,1} 2^{-j} \Delta_\chi} | \e^{-2^{-n-2}\Delta_\chi}  g_n| \\
&\lesssim 2^{j(m+\frac{1}{2})} \e^{-(a_{m,1}2^{-j }+ 2^{-n-2})\Delta_\chi} |g_n|,
\end{aligned}
\end{equation}
but also $a_{2m+1}>0$ such that
\begin{equation}\label{eqpermaxsup2}
\begin{aligned}
|\Delta_\chi^{m} \e^{-2^{-j}\Delta_\chi} X_i f_n |  
&\lesssim \e^{- 2^{-j} \Delta_\chi} | \Delta_\chi^{m} X_i \e^{-(2^{-n-2})\Delta_\chi}  g_n| \\
&\lesssim 2^{n(m+\frac{1}{2})} \e^{-(2^{-j }+ a_{2m+1}2^{-n-2})\Delta_\chi} |g_n|.
\end{aligned}
\end{equation}
Since for $a'_{m} = \min(a_{m,1},1/4, a_{2m+1}/4)$ and for $a_m = \max(a_{m,1},1, a_{2m+1}/4)$ 
\begin{align*}
a_{m,1}2^{-j }+ 2^{-n-2} &\in  (a'_{m} (2^{-j }+ 2^{-n}), a_m (2^{-j }+ 2^{-n})),\\
 2^{-j }+ a_{2m+1}2^{-n-2}& \in (a'_{m} (2^{-j }+ 2^{-n}), a_m (2^{-j }+ 2^{-n})),
\end{align*}
there exists a constant $c>0$ depending only on $m$ such that
\[
|\Delta_\chi^{m} \e^{-2^{-j}\Delta_\chi} X_i f_n |   \lesssim 2^{(m+\frac{1}{2}) \min(n,j)}\e^{-c(2^{-j}+ 2^{-n})\Delta_\chi} |g_n|.
\]
Let $q<\infty$. By~\eqref{upsilondiscreto} and Proposition~\ref{contCRTN},
\begin{align*}
 \mathscr{F}_{\beta}^{p,q}\left(X_i \sum_{n=1}^{\infty} f_n\right)^q 
&\lesssim \Bigg\| \left( \sum_{j=0}^\infty \left( 2^{j\frac{\beta}{2}} \sum_{n=1}^\infty  |W^{(m)}_{2^{-j}} X_i f_n | \right)^q \right)^{1/q}  \Bigg\|_{L^p} \\
& \lesssim  \Bigg\|\left( \sum_{j=0}^\infty \left( 2^{-j(m-\frac{\beta}{2})}  \e^{-c 2^{-j}\Delta_\chi}\sum_{n=1}^\infty 2^{\min(n,j)(m+\frac{1}{2})} \e^{-c2^{-n}\Delta_\chi} |g_n| \right)^q \right)^{1/q}  \Bigg\|_{L^p} \\
& \lesssim \Bigg\| \left( \sum_{j=0}^\infty \left( 2^{-j(m-\frac{\beta}{2})} \sum_{n=1}^\infty 2^{\min(n,j)(m+\frac{1}{2})} \e^{-c2^{-n}\Delta_\chi} |g_n| \right)^q \right)^{1/q}   \Bigg\|_{L^p},
\end{align*}
and by Lemma~\ref{lemmasomme}, the last term of these inequalities is controlled by
\begin{align*}
\Bigg\|\left( \sum_{n=0}^\infty \left( 2^{n \frac{\beta+1}{2}} \e^{-c2^{-n}\Delta_\chi} |g_n| \right)^q \right)^{1/q}   \Bigg\|_{L^p}  \lesssim  \Bigg\|\left( \sum_{n=0}^\infty \left( 2^{n \frac{\beta+ 1}{2}} |g_n| \right)^q \right)^{1/q}  \Bigg\|_{L^p},
\end{align*}
by Proposition~\ref{contCRTN} again. Observe now that by definition of $g_n$ and by Lemma~\ref{pointwiseestheat2}, there exists $c>0$ such that
\[
|g_n| \lesssim \e^{- c2^{-n}\Delta_\chi} \int_{2^{-n}}^{2^{-n+1}}  |W^{(\overline{m})}_{t/2} f| \, \frac{\dd t}{t}.
\]
Hence, by Proposition~\ref{contCRTN} and Jensen's inequality,
\begin{align*}
 \Bigg\| \left( \sum_{j=0}^\infty \left( 2^{n \frac{\beta+ 1}{2}} |g_n| \right)^q \right)^{1/q}   \Bigg\|_{L^p} & 
  \lesssim \Bigg\|\left( \sum_{j=0}^\infty \left( 2^{n \frac{\beta+ 1}{2}}   \int_{2^{-n}}^{2^{-n+1}}  |W^{(\overline{m})}_{t/2} f| \, \frac{\dd t}{t} \right)^q \right)^{1/q}   \Bigg\|_{L^p}\\
  & \lesssim \Bigg\| \left(  \int_{0}^1 \left( t^{- \frac{\beta+ 1}{2}}   |W^{(\overline{m})}_{t/2} f|\right)^{q} \, \frac{\dd t}{t} \right)^{1/q} \Bigg\|_{L^p}\\
  & \lesssim \mathscr{F}^{p,q}_{\beta+1} (f).
\end{align*}
We are then left with estimating the term $\mathscr{F}^{p,q}_{\beta}(\sum_{k=0}^{\overline{m}-1} X_i W^{(k)}_1 f)$. By Lemma~\ref{pointwiseestheat2} and Proposition~\ref{CRTNintegrale}, for every $k\in \{0,\dots, \overline{m}-1\}$
\begin{align*}
 \left\| \left(\int_0^1 \left( t^{-\frac{\beta}{2}} \left| W^{(\overline{m})}_t X_i  W^{(k)}_1 f \right| \right)^q \, \frac{\dd t}{t} \right)^{1/q}\right\|_{L^p}
&  \lesssim \left\| \left(\int_0^1 \left( t^{\overline{m}-\frac{\beta}{2}} \e^{-c\Delta_\chi} |f| \right)^q \, \frac{\dd t}{t} \right)^{1/q}\right\|_{L^p} \\
& \lesssim\|f\|_{L^p}.
\end{align*}
This completes the proof of the claim~\eqref{claimalphaTL} and thus that of Theorem~\ref{teo_recursive} in the case when $q<\infty$. The case $q=\infty$ is left to the reader.
\end{proof}

\subsection{Characterizations by differences}
One can also characterize Besov and Triebel--Lizorkin norms by means of finite differences. For $x,y\in G$, let
\[
\mathrm{D}_yf(x)= f(xy^{-1})-f(x).
\]
For $p\in [1,\infty]$ and $q\in [1,\infty)$, we define
\begin{equation*}
\mathscr{A}^{p,q}_{\alpha}(f)\coloneqq \left(\int_{|y|\leq 1} \left(\frac{\|\mathrm{D}_y f\|_{L^p(\mu_\chi)}}{|y|^\alpha}\right) ^q \, \frac{\dd \rho(y)}{V(|y|)}\right)^{1/q},
\end{equation*}
and
\begin{equation*}
\mathscr{S}^{\loc, q}_{\alpha}f(x) =\left(\int_0^1\left( \frac{1}{u^{\alpha} V(u)}\int_{|y|<u}|\mathrm{D}_yf(x)|\, \dd\rho(y) \right)^q\, \frac{\dd u}{u}\right)^{1/q}.
\end{equation*}
We also define
\begin{equation*}
\mathscr{A}^{p,\infty}_{\alpha}(f)\coloneqq \sup_{|y|\leq 1} \frac{\|\mathrm{D}_y f\|_{L^p(\mu_\chi)}}{|y|^\alpha}.
\end{equation*}
\begin{theorem}\label{teo_functional}
Let $\alpha \in (0,1)$.
\begin{itemize}
\item[(i)] If $p,q\in [1,\infty]$, then
\[
\|f\|_{B_\alpha^{p,q}(\mu_\chi)}\approx \mathscr{A}^{p,q}_{\alpha}(f)+ \|f\|_{L^p(\mu_\chi)}.
\]
\item[(ii)] If $p,q\in (1,\infty)$, then
\[
\|f\|_{F^{p,q}_\alpha(\mu_\chi)} \approx    \|\mathscr{S}^{\loc, q}_{\alpha}f\|_{L^p(\mu_\chi)}+\|f\|_{L^p(\mu_\chi)}.
\]
\end{itemize}
\end{theorem}
The proofs of~(i) and~(ii) can be obtained by suitably adapting the proofs of~\cite[Theorem 1.16]{Feneuil} and~\cite[Theorem 1.3~(i)]{PV} respectively. We omit the details, which are contained in~\cite[Section 3]{BPVsurvey}.

\section{Comparison Theorems}\label{sec:CT}

In this section, we establish embedding properties of Besov and Triebel--Lizorkin spaces, whose Euclidean counterparts can be found in~\cite[Proposition 2, p.\ 47 and Theorem p.\ 129]{TriebelTFS}. We begin by observing that if $X_\alpha^{p,q}(\mu_\chi)$ is either $B_\alpha^{p,q}(\mu_\chi)$ or $F_\alpha^{p,q}(\mu_\chi)$, then embeddings of the form
\[
X_\alpha^{p,q}(\mu_\chi) \hookrightarrow X_{\beta}^{r,s}(\mu_\chi), \qquad X_\alpha^{p,q}(\mu_\chi) \hookrightarrow L^{r}(\mu_\chi)
\]
may hold only if either $p=r$ or $p\neq r$ and $\mu_\chi =\lambda$, by a translation-invariance argument analogous to that of~\cite[Section 4]{BPTV}. We also recall that for $p\in (1,\infty)$ and $\alpha\geq 0$, the Sobolev space $L^p_\alpha(\mu_\chi)$ is defined by means of the norm (see~\cite[Section 3]{BPTV})
\begin{equation}\label{equivtranslation}
\| f\|_{L^p_{\alpha}(\mu_\chi)} = \| f\|_{L^p(\mu_\chi)} + \| \Delta_\chi^{\alpha/2} f\|_{L^p(\mu_\chi)}  \approx \| (\Delta_\chi + I)^{\alpha/2}f\|_{L^p(\mu_\chi)}.
\end{equation}
The following theorem concerns Besov spaces.
\begin{theorem}\label{teo_embeddings}
The following embeddings hold.
\begin{itemize}
\item[(i)] Let $p,q,q_1\in [1,\infty]$ and $\alpha,\alpha_1 \geq 0$. Then 
\[
B_\alpha^{p,q} (\mu_\chi)\hookrightarrow B_{\alpha_1}^{p,q_1}(\mu_\chi)
\]
if either $\alpha_1<\alpha$ or $\alpha_1=\alpha$ and $q_1\geq q$.
\item[(ii)] Let $1\leq p_0< p_1 \leq \infty$, $q\in [1,\infty]$ and $\alpha_0 \geq \alpha_1 \geq 0$. If $\frac{d}{p_1}- \alpha_1 = \frac{d}{p_0}-\alpha_0$, then 
\[B_{\alpha_0}^{p,q}(\lambda) \hookrightarrow B_{\alpha_1}^{p_1,q}(\lambda).\]
\item[(iii)] Let $p\in [1,\infty]$. Then 
\[
B^{p,1}_{d/p} (\lambda) \hookrightarrow L^\infty.
\] 
Moreover, if $q\in [1,\infty]$ and $\alpha>d/p$, then 
\[
B^{p,q}_\alpha(\lambda) \hookrightarrow L^\infty.
\]
\end{itemize}
\end{theorem}
\begin{proof}
We first consider (i). If $\alpha_1<\alpha$, the embedding is a consequence of H\"older's inequality. If $\alpha_1=\alpha$, it is a consequence of Theorem~\ref{teo-equiv2} and the inclusions of the $\ell^q$ spaces.

We now prove (ii). Let $t_0 \in (0,1)$, $m> \alpha/2$ be an integer and $q<\infty$. By Lemma~\ref{normaetdeltaq},  
\[
\| \e^{-t_0 \Ls } f\|_{L^{p_1}(\lambda)} \lesssim \| \e^{-\frac{t_0}{2} \Ls} f\|_{L^{p_0}(\lambda)} ,
\]
and, since $\frac{d}{p_1} - \frac{d}{p_0}= \alpha_1 -\alpha_0$,
\begin{align*}
 \int_0^1 \left( t^{-\frac{\alpha_1}{2}} \| \Ws_t^{(m)} f\|_{L^{p_1}(\lambda)}\right)^q \, \frac{\dd t}{t}& 
 \lesssim \int_0^1 \left( t^{-\frac{\alpha_0}{2}} \| \Ws_{t/2}^{(m)} f\|_{L^{p_0}(\lambda)}\right)^q \, \frac{\dd t}{t}\\
& \lesssim   \int_0^1 \left( t^{-\frac{\alpha_0}{2}} \| \Ws_t^{(m)}  f\|_{L^{p_0}(\lambda)}\right)^q \, \frac{\dd t}{t},
\end{align*}
the second inequality by a change of variables. The conclusion follows by Theorem~\ref{teo-equiv1}. The case $q=\infty$ can be proved analogously.

To prove (iii), let $0<\epsilon< d/p$ and observe that by (ii) and~\eqref{embedLp}
\[
B_{d/p}^{p,1}(\lambda) \hookrightarrow B_\epsilon^{\infty, 1}(\lambda)  \hookrightarrow L^\infty.
\]
If $\alpha>d/p$, then by~(i) and the embedding above
\[
B_\alpha^{p,q}(\lambda)\hookrightarrow B_{d/p}^{p,1}(\lambda )\hookrightarrow L^\infty,
\]
which were the desired embeddings.
\end{proof}

We now turn to comparison theorems for Triebel--Lizorkin spaces.
\begin{theorem}\label{teo_embeddings_TL}
The following embeddings hold.
\begin{itemize}
\item[(i)] Let $p,q,q_1\in [1,\infty]$ and $\alpha,\alpha_1 \geq 0$. Then 
\[
F_\alpha^{p,q} (\mu_\chi)\hookrightarrow F_{\alpha_1}^{p,q_1}(\mu_\chi)
\]
if either $\alpha_1<\alpha$ or $\alpha_1=\alpha$ and $q_1\geq q$.
\item[(ii)] Let $1< p_0< p_1 < \infty$, $q,r\in [1,\infty]$ and $\alpha_0 \geq \alpha_1 \geq 0$. If $\frac{d}{p_1}- \alpha_1 = \frac{d}{p_0}-\alpha_0$, then 
\[F_{\alpha_0}^{p_0,q}(\lambda) \hookrightarrow F_{\alpha_1}^{p_1,r}(\lambda).\]
\item[(iii)] If $p\in (1,\infty)$ and $\alpha\geq 0$, then $F^{p,2}_\alpha(\mu_\chi) = L^p_\alpha(\mu_\chi)$ with equivalence of norms.
\item[(iv)] If $p\in (1,\infty)$, $q\in [1,\infty]$ and $\alpha>d/p$, then
\[
F^{p,q}_\alpha(\lambda) \hookrightarrow L^\infty.
\]
\end{itemize}
\end{theorem}

\begin{proof}
We begin with (i). If $\alpha_1<\alpha$, the embedding is a consequence of H\"older's inequality. If $\alpha_1=\alpha$, it is a consequence of Theorem~\ref{teo-equiv2} and the inclusions of the $\ell^q$ spaces.

We skip the proof of~(ii) for a moment, and prove (iii). By Proposition~\ref{teo-equiv1}, it will be enough to prove that for every $p\in (1,\infty)$, $\alpha\geq 0$, $t_0\in (0,1)$ and $m>\alpha/2$ integer
\begin{equation}\label{embediii}
\|f\|_{L^p_\alpha(\mu_\chi)} \approx \| \e^{-t_0\Delta_\chi}f\|_{L^p(\mu_\chi)} + \Bigg\|\left( \int_0^1 (t^{-\frac{\alpha}{2}} |W^{(m)}_t f|)^2 \, \frac{\dd t}{t} \right)^{1/2} \Bigg\|_{L^p(\mu_\chi)}.
\end{equation}
We first recall that by Littlewood--Paley--Stein theory (see~\cite{Meda} or~\cite[p.\ 6]{PV}),
 \begin{align*}
 \| \Delta_\chi^{\alpha/2}f\|_{L^p} \approx 
\Bigg\| \left( \int_0^\infty \left( t^{-\frac{\alpha}{2}} | W^{(m)}_t f|\right)^2 \, \frac{\dd t}{t}\right)^{1/2} \Bigg\|_{L^p}.
\end{align*}
The inequality $\gtrsim$ of~\eqref{embediii} follows at once, since
\[
\| \e^{-t_0\Delta_\chi}f\|_{L^p} \lesssim \| f\|_{L^p}
\]
and
\[
 \int_0^1 \left(t^{-\frac{\alpha}{2}} |W^{(m)}_t f|\right)^2 \, \frac{\dd t}{t}   \lesssim \int_0^\infty \left( t^{-\frac{\alpha}{2}} |W^{(m)}_t f|\right)^2 \, \frac{\dd t}{t} .
\]
To prove the inequality $\lesssim$ of~\eqref{embediii}, observe that
\begin{align*}
\int_1^\infty \left( t^{-\frac{\alpha}{2}} |W^{(m)}_t f|\right)^2 \, \frac{\dd t}{t} 
& \leq \int_1^\infty \left( t^{m} |\Delta_\chi^{m} \e^{-(t-t_0)\Delta_\chi} \e^{-t_0\Delta_\chi}f|\right)^2 \, \frac{\dd t}{t}\\
& \lesssim \int_0^\infty \left(|W^{(m)}_t \e^{-t_0\Delta_\chi}f|\right)^2 \, \frac{\dd t}{t},
\end{align*}
the last inequality by a change of variables in the integral. Thus, again by Littlewood--Paley--Stein theory
\begin{align*}
\left\| \left( \int_1^\infty \left( t^{-\frac{\alpha}{2}} | W^{(m)}_t f|\right)^2 \, \frac{\dd t}{t}\right)^{1/2} \right\|_{L^p}  \lesssim \|\e^{-t_0\Delta_\chi}f\|_{L^p}.
\end{align*}
This proves that
\begin{align}\label{Sob4}
\|\Delta_\chi^{\alpha/2} f\|_{L^p} \lesssim \|\e^{-t_0\Delta_\chi}f\|_{L^p} +  \left\|\left( \int_0^1 \left(t^{-\frac{\alpha}{2}} |W^{(m)}_t f|\right)^2 \, \frac{\dd t}{t} \right)^{1/2} \right\|_{L^p},
\end{align}
which in particular implies that for every $m>0$ one has
\begin{align}\label{Sob5}
\| f\|_{L^p} & \lesssim \|\e^{-t_0\Delta_\chi}f\|_{L^p} +  \left\|\left( \int_0^1 \left( |W^{(m)}_t f|\right)^2 \, \frac{\dd t}{t} \right)^{1/2} \right\|_{L^p}.
\end{align}
It remains to observe that for every $\alpha \geq 0$
\[
 \int_0^1 \left( |W^{(m)}_t f|\right)^2 \, \frac{\dd t}{t}  \lesssim  \int_0^1 \left(t^{-\frac{\alpha}{2}} |W^{(m)}_t  f|\right)^2 \, \frac{\dd t}{t}
\]
which together with~\eqref{Sob4} and~\eqref{Sob5} proves the inequality $\lesssim$ of~\eqref{embediii}.

We now prove (iv). Let $\alpha> d/p$ and $\beta$ be such that $d/p<\beta<\alpha$. Then by (i),~(iii) and the embeddings of Sobolev spaces (see~\cite[Theorems 1.1 and 4.4]{BPTV})
\[
F^{p,q}_\alpha(\lambda) \hookrightarrow F^{p,2}_\beta(\lambda) = L^p_\beta(\lambda) \hookrightarrow L^\infty.
\]
It remains to prove (ii). Observe that by~(i) it is enough to prove that
\begin{equation}\label{embeddingii}
F^{p_0,\infty}_{\alpha_0}(\lambda)\hookrightarrow F^{p_1,1}_{\alpha_1}(\lambda).
\end{equation}
Let $m_0>\frac{\alpha_0}{2}>\frac{\alpha_1}{2}$ and to simplify the notation, define the operators 
\[
\mathcal{T}_j^0 = 2^{j \frac{\alpha_0}{2}} \Ws^{(m_0)}_{2^{-j}}, \qquad \mathcal{T}_j^1 =2^{j \frac{\alpha_1}{2}} \Ws^{(m_0)}_{2^{-j}},
\]
and observe that
\[
\| f\|_{F^{p_0,\infty}_{\alpha_0}(\lambda)} = \big\| \sup_{j\in \N} |\mathcal{T}_j^0 f| | \big\|_{L^{p_0}(\lambda)}, \qquad \| f\|_{F^{p_1,1}_{\alpha_1}(\lambda)} = \Big\| \sum_{j=0}^\infty| \mathcal{T}_j^1 f|| \Big\|_{L^{p_1}(\lambda)}.
\]
Without loss of generality, we may assume that $\| f\|_{F^{p_0,\infty}_{\alpha_0}(\lambda)}=1$. By Lemma~\ref{normaetdeltaq},
\begin{align*}
\| \mathcal{T}_j^1 f\|_{L^\infty} &
= 2^{-j \frac{\alpha_0-\alpha_1}{2}}  \| 2^{-j (m_0-\frac{\alpha_0}{2})} \e^{-2^{-{j-1}}\Ls} \e^{-2^{-{j-1}}\Ls}\Ls^{m_0}f\|_{L^\infty}\\
&\lesssim 2^{-j \frac{\alpha_0-\alpha_1}{2} + j\frac{d}{2p_0}}\| \mathcal{T}_{j-1}^0 f\|_{L^{p_0}(\lambda)},
\end{align*}
so that, for every $K\in \N$,
\begin{equation}\label{eq0K}
\sum_{j=0}^K |\mathcal{T}_j^1 f| \lesssim \sum_{j=0}^K 2^{-j \frac{\alpha_0-\alpha_1}{2} + j\frac{d}{2p_0}}.
\end{equation}
Moreover, one has
\begin{equation}\label{eqK+1infty}
\sum_{j=K+1}^\infty |\mathcal{T}_j^1 f| =  \sum_{j=K+1}^\infty 2^{-j \frac{\alpha_0 -\alpha_1}{2}} |\mathcal{T}_j^0 f| \lesssim 2^{-K\frac{\alpha_0 -\alpha_1}{2}} \sup_{j\in \N} |\mathcal{T}_j^0f|.
\end{equation}
Now,
\begin{align*}
\| f\|_{F^{p_1,1}_{\alpha_1}(\lambda)}^{p_1} 
& = p_1\int_0^\infty t^{p_1-1} \lambda \Bigg( \Bigg\{ x \colon \sum_{j=0}^\infty |\mathcal{T}_j^1 f(x)|>t \Bigg\} \Bigg) \, \dd t = \int_0^1 \dots \, \dd t + \int_1^\infty\dots \, \dd t.
\end{align*}
By~\eqref{eqK+1infty} with $K=-1$, there exists $C>0$ such that
\[
\Bigg\{ x \colon \sum_{j=0}^\infty |\mathcal{T}_j^1 f(x)|>t \Bigg\} \subset \Bigg\{ x \colon \sup_{j\in \N} |\mathcal{T}_j^0 f(x)| > C t \Bigg\}
\]
and hence
\begin{align*}
\int_0^{1} t^{p_1} \lambda \Bigg( \Bigg\{ x \colon \sum_{j=0}^\infty |\mathcal{T}_j^1 f(x)|>t \Bigg\} \Bigg) \, \frac{\dd t}{t} 
& \lesssim \int_0^1  t^{p_0} \lambda \Bigg( \Bigg\{ x \colon  \sup_{j\in \N} |\mathcal{T}_j^0 f(x)|>C t \Bigg\} \Bigg) \, \frac{\dd t}{t}\\
& \lesssim  \Big\| \sup_{j\in \N} |\mathcal{T}_j^0 f| \Big\|_{L^{p_0}(\lambda)}^{p_0}\\
& \lesssim 1.
\end{align*}
Observe now that
\[
\Bigg\{ x \colon \sum_{j=0}^\infty |\mathcal{T}_j^1 f(x)|>t \Bigg\} \subset \Bigg\{ x \colon \sum_{j=K(t)+1}^\infty |\mathcal{T}_j^1 f(x)|>\frac{t}{2}\Bigg\}
\]
where $K=K(t)$ is the largest integer such that
\[
 \sum_{j=0}^K 2^{\frac{j}{2}( \alpha_1-\alpha_0+ \frac{d}{p_0})} = \sum_{j=0}^K 2^{\frac{jd}{2p_1}} <\frac{t}{2}.
\]
In other words, $K=K(t)$ is such that $2^{\frac{Kd}{2p_1} } \approx t$. By~\eqref{eqK+1infty} 
\[
 \Bigg\{ x \colon \sum_{j=K(t)+1}^\infty |\mathcal{T}_j^1 f(x)|>\frac{t}{2}\Bigg\} \subset \Bigg\{ x \colon \sup_{j\in \N} |\mathcal{T}_j^0f(x)|>C t 2^{-K\frac{\alpha_1 -\alpha_0}{2}}\Bigg\},
\]
and
\[
t 2^{-K\frac{\alpha_1 -\alpha_0}{2}} \approx t^{\frac{p_1}{p_0}}.
\]
Then,
\begin{align*}
\int_1^\infty  t^{p_1-1} \lambda \Bigg(\Bigg\{ x \colon \sum_{j=0}^\infty |\mathcal{T}_j^1 f(x)|>t \Bigg\} \Bigg) \, \dd t
& \lesssim \int_1^\infty  t^{p_1-1} \lambda \Bigg( \Bigg\{ x \colon  \sup_{j\in \N} |\mathcal{T}_j^0 f(x)|>C t^{\frac{p_1}{p_0}}  \Bigg\} \Bigg) \, \dd t\\
& \lesssim \int_0^\infty  s^{p_0-1} \lambda \Bigg( \Bigg\{ x \colon  \sup_{j\in \N} |\mathcal{T}_j^0 f(x)|> s\  \Bigg\} \Bigg) \, \dd s\\
&  \lesssim 1.
\end{align*}
The proof is complete.
\end{proof}

The following result is the counterpart of~\cite[Section 2.3.2, Proposition 2]{TriebelTFS} in the Euclidean context. It compares Besov with Triebel--Lizorkin spaces.
\begin{theorem}\label{teo_embeddingp2}
Let $p\in (1,\infty)$, $q\in [1,\infty]$ and $\alpha\geq 0$. Then
\[
B^{p,\min(p,q)}_\alpha(\mu_\chi) \hookrightarrow F^{p,q}_\alpha(\mu_\chi)\hookrightarrow B_\alpha^{p,\max(p,q)} (\mu_\chi).
\]
\end{theorem}
\begin{proof}
Let first $p\geq q$. Then, since $\ell^q \hookrightarrow \ell^p$,
\begin{align*}
 \left(\sum_{j=1}^\infty \left(2^{j \frac{\alpha}{2}} \| W^{(m)}_{2^{-j}} f\|_{L^p} \right)^p \right)^{1/p}
 &  = \left( \int_G \| 2^{j\frac{\alpha}{2}}W^{(m)}_{2^{-j}} f \|_{\ell^p}^p \, \dd \mu_\chi \right)^{1/p}\\
&\leq  \left( \int_G \| 2^{j \frac{\alpha}{2}} W^{(m)}_{2^{-j}} f \|_{\ell^q}^p \, \dd \mu_\chi \right)^{1/p}\\
& =  \left\| \sum_{j=1}^\infty \left(2^{j\frac{\alpha}{2}} | W^{(m)}_{2^{-j}} f| \right)^q \right\|_{L^{p/q}}^{1/q},
\end{align*}
and by the triangle inequality in $L^{p/q}(\mu_\chi)$,
\begin{align*}
\left\| \sum_{j=1}^\infty \left(2^{j\frac{\alpha}{2}} |W^{(m)}_{2^{-j}}  f| \right)^q \right\|_{L^{p/q}}^{1/q}
& \leq    \left( \sum_{j=1}^\infty 2^{qj\frac{\alpha}{2}} \| (W^{(m)}_{2^{-j}}  f)^q \|_{L^{p/q}}\right)^{1/q}\\
& = \left( \sum_{j=1}^\infty \left(2^{j\frac{\alpha}{2}} \| W^{(m)}_{2^{-j}}  f \|_{L^{p}}\right)^q\right)^{1/q}.
\end{align*}
The conclusion follows by Theorem~\ref{teo-equiv2}. Similarly, if $p< q<\infty$ then
\begin{align*}
\left( \sum_{j=1}^\infty \left(2^{j \frac{\alpha}{2}} \| W^{(m)}_{2^{-j}}  f \|_{L^{p}}\right)^q\right)^{1/q}
& = \left\| \int_G 2^{jp \frac{\alpha}{2}} |W^{(m)}_{2^{-j}}  f|^p \, \dd \mu_\chi \right\|_{\ell^{q/p}}^{1/p}\\
& \leq  \left(\int_G \| 2^{jp\frac{\alpha}{2}} |W^{(m)}_{2^{-j}}  f|^p \|_{\ell^{q/p}} \, \dd \mu_\chi \right)^{1/p}\\
& =  \left(\int_G \left( \sum_{j=0}^\infty \left( 2^{j \frac{\alpha}{2}} |W^{(m)}_{2^{-j}}  f|\right)^q \right)^{p/q} \, \dd \mu_\chi \right)^{1/p}\\
& \leq \left(\int_G \sum_{j=0}^\infty \left( 2^{j \frac{\alpha}{2}} |W^{(m)}_{2^{-j}}  f|\right)^p  \, \dd \mu_\chi \right)^{1/p}\\
& = \left(\sum_{j=0}^\infty \left( 2^{j\frac{\alpha}{2}} \|W^{(m)}_{2^{-j}}  f\|_{L^p}\right)^p\right)^{1/p},
\end{align*}
and the conclusion follows. The proof in the case $p< q=\infty$ is easier and omitted.
\end{proof}

\section{Complex interpolation}\label{sec:IP}
In this section we describe the complex interpolation properties of Besov and Triebel--Lizorkin spaces. Given a compatible couple of Banach spaces $A_0$ and $A_1$, we denote with $(A_0,A_1)_{[\theta]}$ the intermediate space of index $\theta \in (0,1)$ in the complex method (see~\cite{BerghLofstrom}). We recall for future convenience that by~\cite[Theorem 4.7.1, p.\ 102]{BerghLofstrom}  and~\cite[p.\ 49]{BerghLofstrom}, one has
\begin{equation}\label{interpinequality}
\| a\|_{(A_0, A_1)_{[\theta]} } \lesssim \| a\|_{A_0}^{1-\theta} \|a\|_{A_1}^{\theta}
\end{equation}
for every $\theta \in (0,1)$.

\begin{theorem}\label{interpolation}
Let $\alpha_0,\alpha_1\geq 0$,  $\theta \in (0,1)$, $\alpha_\theta = (1-\theta)\alpha_0 + \theta \alpha_1$ and $q_0,q_1\in [1,\infty]$.
\begin{itemize}
\item[(i)] If $ p_0,p_1 \in [1,\infty]$, then 
\[(B^{p_0,q_0}_{\alpha_0}(\mu_\chi), B^{p_1,q_1}_{\alpha_1}(\mu_\chi))_{[\theta]} = B^{p_\theta,q_\theta}_{\alpha_\theta}(\mu_\chi),\]
where $\frac{1}{p_\theta} = \frac{1-\theta}{p_0} + \frac{\theta}{p_1}$ and $\frac{1}{q_\theta} = \frac{1-\theta}{q_0} + \frac{\theta}{q_1}$.
\item[(ii)] If $p_0,p_1 \in (1,\infty)$, then 
\[
(F^{p_0,q_0}_{\alpha_0}(\mu_\chi), F^{p_1,q_1}_{\alpha_1}(\mu_\chi))_{[\theta]} = F^{p_\theta,q_\theta}_{\alpha_\theta}(\mu_\chi),\]
where $p_\theta$ and $q_\theta$ are as above.
\end{itemize}
\end{theorem}

\begin{proof}
The proof is inspired to~\cite[Theorem 6.4.3]{BerghLofstrom}. We prove only (ii), for the proof of (i) follows the same steps and is easier  in some respects. To prove~(i), one may also adapt the proof of~\cite[Corollary 4.7]{Feneuil}.

To prove (ii), it is enough to prove that the spaces $F^{p,q}_\alpha(\mu_\chi)$ are retracts of 
\[
L^p(\ell_\alpha^q,\mu_\chi) = \Bigg\{ u=(u_j)_{j\in \N} \colon  \|u\|_{L^p(\ell_\alpha^q, \mu_\chi)}=  \Bigg\| \left( \sum_{j=0}^\infty \left(2^{j\frac{\alpha}{2}} |u_j|\right)^q \right)^{1/q}\Bigg\|_{L^p(\mu_\chi)}<\infty \Bigg\},
\]
with the obvious modification when $q=\infty$. The result will then follow by~\cite[Theorem 6.4.2]{BerghLofstrom} and the complex interpolation properties of the spaces $L^p(\ell_\alpha^q,\mu_\chi)$ (see~\cite[Theorem p.128]{TriebelInterp}). We recall that a space $Y$ is called a retract of $X$ if there exist two bounded linear operators $\mathcal J\colon Y\to X$ and $\mathcal P\colon X\to Y$ such that $\mathcal P \circ \mathcal J$ is the identity on $Y$ (see~\cite[Definition 6.4.1]{BerghLofstrom}).

Let $m= [\frac{\alpha}2]+1$. Define the functional $\mathcal J$  on $F^{p,q}_\alpha(\mu_\chi)$ by $\mathcal J f = \left((\mathcal J f)_j\right)_{j\in \N}$ where
\[
(\mathcal J f)_0 = \e^{-\frac12\Delta_\chi} f, \qquad (\mathcal J f)_j = 2^{m}W^{(m)}_{2^{-j-1}}  f\quad \mbox{ if } j\geq 1,
\]
and $\mathcal P$ on $L^p(\ell_\alpha^q, \mu_\chi)$ by
\begin{align*}
\mathcal P u 
&= \sum_{k=0}^{2m-1} \frac{1}{k!} \Delta_\chi^k \e^{-\frac12\Delta_\chi} u_0 + \frac{1}{(2m-1)!}\sum_{j= 1}^\infty 2^{jm} \int_{2^{-j}}^{2^{-j+1}} t^{2m}\Delta_\chi^m \e^{-(t-2^{-j-1})\Delta_\chi} u_j \, \frac{\dd t}{t} \\
&\eqqcolon \mathcal P_1 u + \mathcal P_2 u. 
\end{align*}
By~\eqref{LPD}, $\mathcal P \circ \mathcal J  = \mathrm{Id}_{F^{p,q}_\alpha}$. Moreover, $\mathcal J$ is bounded from $F^{p,q}_\alpha(\mu_\chi)$ to $L^p(\ell_\alpha^q,\mu_\chi)$ by Theorem~\ref{teo-equiv2}. Thus, it remains to prove that $\mathcal P$ is bounded from $L^p(\ell_\alpha^q,\mu_\chi)$ to $F^{p,q}_\alpha(\mu_\chi)$. 

We assume $q<\infty$. By Lemma~\ref{pointwiseestheat2}, one gets 
\begin{align*}
\| \mathcal{P}_1 u \|_{F^{p,q}_\alpha}
&\lesssim \sum_{k=0}^{2m-1} \Bigg\| \left( \int_0^1 \left( t^{-\frac{\alpha}{2}} | W^{(m)}_{t} W^{(k)}_{1/2}    u_0| \right)^q \, \frac{\dd t}{t} \right)^{1/q} \Bigg\|_{L^p}  + \| \e^{-\frac{1}{2}\Delta_\chi} \mathcal{P}_1 u\|_{L^p} \\
&\lesssim \|u_0\|_{L^p}\\
& \lesssim \| u\|_{L^p(\ell_\alpha^q)}.
\end{align*}
By Lemma~\ref{pointwiseestheat2}, Proposition~\ref{contCRTN} and H\"older's inequality we have
\begin{align*}
 \|\e^{-\frac{1}{2}\Delta_\chi} \mathcal P_2u \|_{L^p} &
  \lesssim \left\| \sum_{j=1}^\infty 2^{jm}  \int_{2^{-j}}^{2^{-j+1}} t^{2m} |\Delta_\chi^m \e^{-(t-2^{-j-1})\Delta_\chi}\e^{-\frac12\Delta_\chi}  u_j| \, \frac{\dd t}{t} \right\|_{L^p}\\
  &\lesssim \left\| \sum_{j=1}^\infty 2^{-jm}  \e^{-c\Delta_\chi}| u_j| \right\|_{L^p}\\
  &\lesssim \left\| \sum_{j=1}^\infty 2^{-jm} | u_j| \right\|_{L^p}\\
  & \lesssim\|u\|_{L^p(\ell_\alpha^q)}.
\end{align*} 
Now, we use~\eqref{upsilondiscreto}, Lemma~\ref{pointwiseestheat2} and Proposition~\ref{contCRTN}, which yield
\begin{align*}
\mathscr{F}^{p,q}_\alpha(\mathcal{P}_2u) 
&\lesssim \Bigg\| \left( \sum_{k=1}^\infty \left( 2^{-k(m-\frac{\alpha}{2})} \sum_{j=1}^\infty 2^{jm} \int_{2^{-j}}^{2^{-j+1}} t^{2m}|\Delta_\chi^{2m} \e^{-(t-2^{-j-1}+2^{-k})\Delta_\chi} u_j |\, \frac{\dd t}{t} \right)^q \right)^{1/q}  \Bigg\|_{L^{p}}   \\
&\lesssim  \Bigg\| \left( \sum_{k=1}^\infty \left( 2^{-k(m-\frac{\alpha}{2})} \sum_{j=1}^\infty 2^{jm} \int_{2^{-j}}^{2^{-j+1}}  \frac{t^{2m}}{(2^{-j}+2^{-k})^{2m}} \e^{-c(2^{-j}+2^{-k})\Delta_\chi}| u_j |\, \frac{\dd t}{t} \right)^q \right)^{1/q} \Bigg\|_{L^{p}}   \\
&\lesssim  \Bigg\| \left( \sum_{k=1}^\infty \left( 2^{-k(m-\frac{\alpha}{2})} \sum_{j=1}^\infty    \frac{2^{-jm}}{(2^{-j}+2^{-k})^{2m}} \e^{-c2^{-j}\Delta_\chi}| u_j | \right)^q \right)^{1/q}  \Bigg\|_{L^{p}} ,
\end{align*}
hence by Lemma~\ref{lemmasomme}
\begin{align*}
\mathscr{F}^{p,q}_\alpha(\mathcal{P}_2u)  &\lesssim  \Bigg\| \left(\sum_{j\geq 1} (2^{j\frac{\alpha}{2}} \e^{-c2^{-j}\Delta_\chi}| u_j |)^q \right)^{1/q}  \Bigg\|_{L^{p}} \lesssim \Bigg\| \left(\sum_{j\geq 1} (2^{j\frac{\alpha}{2}} | u_j |)^q\right)^{1/q}  \Bigg\|_{L^{p}} = \| u\|_{L^p(\ell_\alpha^q, \mu_\chi)}.
\end{align*}
Leaving the case $q=\infty$ to the reader, this concludes the proof.
\end{proof}

\section{Algebra Properties}\label{sec:AP}
In this final section we establish algebra properties of Besov and Triebel--Lizorkin spaces. In particular, we prove the following.
\begin{theorem} \label{teo_algebra}
Let  $\alpha>0$ and $p,p_1,p_2,p_3,p_4,q\in [1,\infty]$ such that
\[\frac{1}{p_1} + \frac{1}{p_2} = \frac{1}{p_3} + \frac{1}{p_4} = \frac{1}{p}.\]
\begin{itemize}
\item[(i)] If $f\in B^{p_1,q}_\alpha(\mu_\chi) \cap L^{p_3}(\mu_\chi)$ and $g \in B^{p_4,q}_\alpha(\mu_\chi) \cap L^{p_2}(\mu_\chi)$, then
\begin{equation}\label{Besov_product}
\|fg\|_{B^{p,q}_\alpha(\mu_\chi)} \lesssim \|f\|_{B^{p_1,q}_\alpha(\mu_\chi)}\|g\|_{L^{p_2}(\mu_\chi)} + \|f\|_{L^{p_3}(\mu_\chi)} \|g\|_{B^{p_4,q}_\alpha(\mu_\chi)}.
\end{equation}
In particular, $B^{p,q}_\alpha(\mu_\chi) \cap L^\infty$ is an algebra under pointwise multiplication.
\item[(ii)] Let $p,p_1,p_4,\in (1,\infty)$ and $p_2,p_3 \in (1,\infty]$. If $f\in F^{p_1,q}_\alpha(\mu_\chi) \cap L^{p_3}(\mu_\chi)$ and $g \in F^{p_4,q}_\alpha(\mu_\chi) \cap L^{p_2}(\mu_\chi)$, then
\begin{equation}\label{TL_product}
\|fg\|_{F^{p,q}_\alpha(\mu_\chi)} \lesssim \|f\|_{F^{p_1,q}_\alpha(\mu_\chi)}\|g\|_{L^{p_2}(\mu_\chi)} + \|f\|_{L^{p_3}(\mu_\chi)} \|g\|_{F^{p_4,q}_\alpha(\mu_\chi)}.
\end{equation}
In particular, $F^{p,q}_\alpha(\mu_\chi) \cap L^\infty$ is an algebra under pointwise multiplication.
\end{itemize}
\end{theorem}

By Theorems~\ref{teo_embeddings}~(iii) and~\ref{teo_embeddings_TL}~(iv) we obtain the following corollary.
\begin{corollary} Let $q\in [1,\infty]$.
\begin{itemize} 
\item[(i)] If $p \in [1,\infty]$ and $\alpha>d/p$, then $B^{p,1}_{d/p}(\lambda)$ and $B^{p,q}_\alpha(\lambda)$ are algebras under pointwise multiplication.
\item[(ii)] If $p\in (1,\infty)$ and $\alpha>d/p$, then $F^{p,q}_\alpha(\lambda)$ is an algebra under pointwise multiplication.
\end{itemize}
\end{corollary}

To prove Theorem~\ref{Besov_product}, we shall use paraproducts, see~\cite{BBR, Feneuil}. The following proposition is essentially~\cite[Proposition 5.2]{Feneuil}, and its proof is exactly the same.

\begin{proposition}\label{paraproduct}
Let $p,q\in [1,\infty]$ such that $\frac{1}{p} + \frac{1}{q} \leq 1$. If $f \in L^p(\mu_\chi)$ and $ g\in L^q(\mu_\chi)$, then 
\[
fg = \Pi_f(g) + \Pi_g(f) + \Pi(f,g) + \sum_{h,k,n=0}^{m-1} \frac{1}{h!k!n!} W^{(h)}_1 [ W^{(k)}_1 f\cdot W^{(n)}_1 g] 
\]
in $\mathcal{S}'(G)$, where
\begin{equation}\label{Pi_fg}
\Pi_f(g) = \sum_{h,k=0}^{m-1} \frac{1}{(m-1)! h! k!} \int_0^1 W^{(h)}_t [W^{(m)}_t  f \cdot W^{(k)}_t g] \, \frac{\dd t}{t},
\end{equation}
and
\begin{equation}\label{Pifg}
\Pi(f,g) = \sum_{h,k=0}^{m-1} \frac{1}{(m-1)! h! k!} \int_0^1 W^{(m)}_t  [W^{(h)}_t  f \cdot W^{(k)}_t g] \, \frac{\dd t}{t}.
\end{equation}
\end{proposition}

We are now ready to prove Theorem~\ref{teo_algebra}.
\begin{proof}[Proof of Theorem~\ref{teo_algebra}]
We prove only (ii), for the proof of (i) follows the same steps. See also the proof of~\cite[Proposition 5.3]{Feneuil}.

We claim that
\begin{equation}\label{eqalg1TL}
 \mathscr{F}^{p,q}_{\alpha}(\Pi_f(g)) \lesssim \|f\|_{F^{p_1,q}_\alpha}\|g\|_{L^{p_2}},
\end{equation}
that
\begin{equation}\label{eqalg3TL}
 \mathscr{F}^{p,q}_{\alpha}(\Pi_g(f)) \lesssim   \|f\|_{L^{p_3}}   \|g\|_{F^{p_4,q}_\alpha},
\end{equation}
and that
\begin{equation}\label{eqalg2TL}
 \mathscr{F}^{p,q}_{\alpha}(\Pi(f,g)) \lesssim \|f\|_{F^{p_1,q}_\alpha}\|g\|_{L^{p_2}} + \|f\|_{L^{p_3}} \|g\|_{F^{p_4,q}_\alpha}.
\end{equation}
Postponing the proof of~\eqref{eqalg1TL},~\eqref{eqalg3TL} and~\eqref{eqalg2TL}, we prove the theorem. By Proposition~\ref{paraproduct} and the claim, it will be enough to prove that
\[
 \|fg\|_{L^p} \lesssim \|f\|_{F^{p_1,q}_\alpha}\|g\|_{L^{p_2}} + \|f\|_{L^{p_3}} \|g\|_{F^{p_4,q}_\alpha}
 \]
and
\[
\| W^{(h)}_1 [W^{(k)}_1 f \cdot W^{(n)}_1 g]\|_{F^{p,q}_\alpha} \lesssim \|f\|_{F^{p_1,q}_\alpha}\|g\|_{L^{p_2}} + \|f\|_{L^{p_3}} \|g\|_{F^{p_4,q}_\alpha}
\]
for every $h,k,n\in \{0,\dots, m-1\}$. The first inequality is a consequence of H\"older's inequality:
\[
\|f g\|_{L^p} \leq \|f\|_{L^{p_1}(\mu_\chi)}\|g\|_{L^{p_2}} \leq \|f\|_{F^{p_1,q}_\alpha} \|g\|_{L^{p_2}}.
\]
In order to prove the second inequality, observe that by Lemma~\ref{pointwiseestheat2} there exists a positive constant $c$ such that
\[
| W^{(h)}_1[W^{(k)}_1 f \cdot W^{(n)}_1 g]| \lesssim \e^{-c \Delta_\chi}|W^{(k)}_1 f \cdot W^{(n)}_1 g|
\]
so that, by Proposition~\ref{CRTNintegrale}, Lemma~\ref{pointwiseestheat3} and H\"older's inequality
\begin{align*}
   \|W^{(h)}_1[W^{(k)}_1 f \cdot W^{(n)}_1 g]]\|_{F^{p,q}_\alpha} & \lesssim \|W^{(k)}_1 f \cdot W^{(n)}_1 g\|_{L^p} \\
& \lesssim \|W^{(k)}_1 f\|_{L^{p_1}} \|W^{(n)}_1 g\|_{L^{p_2}} \\
&\lesssim \|f\|_{L^{p_1}} \|g\|_{L^{p_2}} \\
& \lesssim \|f\|_{F^{p_1,q}_\alpha} \|g\|_{L^{p_2}}.
\end{align*}
Therefore, it remains to prove the claim. We provide the details only when $q<\infty$.

\smallskip

\textit{Step 1.} We prove~\eqref{eqalg1TL} and~\eqref{eqalg3TL}. Let $m = [\alpha/2]+1$. By~\eqref{Pi_fg}
\[
\mathscr{F}^{p,q}_\alpha(\Pi_f(g)) \lesssim \sum_{h,k=0}^{m-1} \left\| \left( \int_0^1 \left( u^{-\frac{\alpha}{2}} \int_0^1 \left|W^{(m)}_u W^{(h)}_t [W^{(m)}_t  f \cdot W^{(k)}_t g]\right| \, \frac{\dd t}{t}  \right)^q \, \frac{\dd u}{u}\right)^{1/q} \right\|_{L^p}.
\]
Thus, let now $h,k \in \{0,\dots, m-1\}$ and $u\in (0,1)$. By Lemma~\ref{pointwiseestheat2}, there exist $a_{2h},a_{2m}>0$ such that
\begin{align*}
|W^{(m)}_u  W^{(h)}_t [ W^{(m)}_t f \cdot  W^{(k)}_t g] | 
& = |W^{(h)}_t  W^{(m)}_u [ W^{(m)}_t f \cdot  W^{(k)}_t g] |  \\
& \lesssim u^m  \e^{-\frac{a_{2h}}{2}t \Delta_\chi} |\Delta_\chi^m \e^{-(\frac t2+u)\Delta_\chi} [W^{(m)}_t f \cdot  W^{(k)}_t g ]|\\
& \lesssim u^m \left(\frac t2+u\right)^{-m} \e^{-\frac{a_{2h}}{2}t \Delta_\chi} \e^{-a_{2m} (\frac t2+u)\Delta_\chi} | W^{(m)}_t f \cdot  W^{(k)}_t g |\\
& \lesssim u^m(u+t)^{-m} \e^{-c (t+u)\Delta_\chi} |W^{(m)}_t f \cdot  W^{(k)}_t g|,
\end{align*}
for some $c>0$. Therefore, by Lemmata~\ref{CRTNintegrale},~\ref{lemmasomme}~(ii), and~\ref{bilinearCRTNint}
\begin{align*}
 &  \left\|  \left( \int_0^1 \left( u^{-\frac{\alpha}{2}}  \int_0^1 |W^{(m)}_u W^{(h)}_t [W^{(m)}_t f \cdot W^{(k)}_t g] |  \, \frac{\dd t}{t} \right)^q \, \frac{\dd u}{u} \right)^{1/q} \right\|_{L^p} \\
& \lesssim \left\|  \left( \int_0^1 \left( u^{m-\frac{\alpha}{2}} \e^{-cu \Delta_\chi} \int_{0}^{1} (u+t)^{-m} \e^{-ct\Delta_\chi}  |W^{(m)}_t f \cdot  W^{(k)}_t g| \, \frac{\dd t}{t}  \right)^q \, \frac{\dd u}{u} \right)^{1/q} \right\|_{L^p} \\
& \lesssim \left\|  \left( \int_0^1 \left( u^{m-\frac{\alpha}{2}} \int_{0}^{1} (u+t)^{-m} \e^{-ct\Delta_\chi}  |W^{(m)}_t f \cdot  W^{(k)}_t g| \, \frac{\dd t}{t}  \right)^q \, \frac{\dd u}{u} \right)^{1/q} \right\|_{L^p} \\
& \lesssim \left\|  \left( \int_0^1 \left( t^{- \frac{\alpha}{2}} \e^{-ct\Delta_\chi}  |W^{(m)}_t f \cdot  W^{(k)}_t g|\right)^q  \, \frac{\dd t}{t}  \right)^{1/q} \right\|_{L^p} \\
& \lesssim \left\|  \left( \int_0^1 \left( t^{- \frac{\alpha}{2}}  |W^{(m)}_t f \cdot  W^{(k)}_t g |\right)^q  \, \frac{\dd t}{t}  \right)^{1/q} \right\|_{L^p}.
 \end{align*}
By  Lemma~\ref{pointwiseestheat2} $ |W^{(k)}_ t g| \lesssim \e^{-c t\Delta_\chi} |g|$. Hence, by Proposition~\ref{CRTNintegrale}, H\"older's inequality and the $L^{p_2}$-boundedness of the local heat maximal function (observe that $p_2>1$) we obtain
\begin{align*}
 \left\|  \left( \int_0^1 \left( t^{- \frac{\alpha}{2}}  |W^{(m)}_t f \cdot  W^{(k)}_t g |\right)^q  \, \frac{\dd t}{t}  \right)^{1/q} \right\|_{L^p}
 &\lesssim \left\|  \left( \int_0^1 \left( t^{- \frac{\alpha}{2}} |W^{(m)}_t f| \cdot \e^{-ct\Delta_\chi} |g| \right)^q  \, \frac{\dd t}{t}  \right)^{1/q} \right\|_{L^p}\\
  & \lesssim \bigg\|\sup_{t\in (0,1)} \e^{-ct\Delta_\chi} |g| \bigg\|_{L^{p_2}} \| f\|_{F^{p_1, q}_\alpha} \\
  & \lesssim \|g\|_{L^{p_2}} \| f\|_{F^{p_1, q}_\alpha}.
\end{align*}
The proof of~\eqref{eqalg1TL} is thus complete. The proof of~\eqref{eqalg3TL} is similar and omitted. 

\smallskip

\textit{Step 2.} We prove~\eqref{eqalg2TL}. By Lemma~\ref{pointwiseestheat2} and the Leibniz rule,
\begin{align*}
&|W^{(m)}_u W^{(m)}_t  [W^{(h)}_t  f \cdot W^{(k)}_t g] |\\
&= u^m |\Delta_\chi^m \e^{-(u+t)\Delta_\chi} (t\Delta_\chi)^m  [W^{(h)}_t  f \cdot W^{(k)}_t g]| \\
& \lesssim u^m(t+u)^{-m} \e^{-a_{2m}(t+u)\Delta_\chi} |(t\Delta_\chi)^m[W^{(h)}_t  f \cdot W^{(k)}_t g]| \\
& \lesssim u^m (t+u)^{-m} \e^{-a_{2m}(t+u)\Delta_\chi} t^{m+h+k} \sum_{i=0}^{2m} \max_{|L|= i+2h} \max_{|J|= 2m+2k-i} | Y_{L}\e^{-t\Delta_\chi} f \cdot Z_{J} \e^{-t\Delta_\chi}g |,
\end{align*}
where $(Y_L, Z_J) = (\Delta_\chi^{h}, \Delta_\chi^{m+k})$ if $i=0$, $(Y_L, Z_J) = (\Delta_\chi^{m+h}, \Delta_\chi^{k})$ if $i=2m$ and $(Y_L, Z_J) = (X_L, X_J)$ otherwise. Thus, after defining 
\[
F(f,g)= \sum_{i=0}^{2m} F_i(f,g), \qquad F_i(f,g)=t^{m+h+k} \max_{|L|= i+2h} \max_{|J|= 2m+2k-i} | Y_{L}\e^{-t\Delta_\chi} f \cdot Z_{J} \e^{-t\Delta_\chi}g |,
\]
by Proposition~\ref{CRTNintegrale} and Lemma~\ref{lemmasomme} we obtain
\begin{align*}
  &\left\|\left( \int_0^1 \left( u^{-\frac{\alpha}{2}} \int_{0}^1|W^{(m)}_u W^{(m)}_t  [W^{(h)}_t  f \cdot W^{(k)}_t g] | \, \frac{\dd t}{t} \right)^q \, \frac{\dd u}{u}\right)^{1/q} \right\|_{L^p}\\
&\qquad \lesssim  \left\|\left( \int_0^1 \left( u^{m-\frac{\alpha}{2}} \e^{-a_{2m} u\Delta_\chi}\int_{0}^1 (t+u)^{-m} \e^{-ct\Delta_\chi}  F(f,g) \, \frac{\dd t}{t} \right)^q \, \frac{\dd u}{u}\right)^{1/q} \right\|_{L^p}\\
&\qquad \lesssim  \left\|\left( \int_0^1 \left( u^{m-\frac{\alpha}{2}} \int_{0}^1 (t+u)^{-m} \e^{-ct\Delta_\chi} F(f,g) \, \frac{\dd t}{t} \right)^q \, \frac{\dd u}{u}\right)^{1/q} \right\|_{L^p}\\
&\qquad  \lesssim \sum_{i=0}^{2m} \left\|\left( \int_0^1 \left( t^{-\frac{\alpha}{2}}  \e^{-ct\Delta_\chi} F_i(f,g)\right)^q \, \frac{\dd t}{t}  \right)^{1/q} \right\|_{L^p}.
\end{align*}
We separate two cases, depending on the values of $i$. The cases $i=0$ or $i=2m$ are symmetric, so that we can assume without loss of generality that $i=0$. Thus,
\begin{align*}
  t^{m+h+k}  \max_{|L|= 2h} \max_{|J|= 2m+2k}  | Y_{L}\e^{-t\Delta_\chi} f \cdot Z_{J} \e^{-t\Delta_\chi}g| = | W^{(h)}_t f | |W^{(m+k)}_t g|.
\end{align*}
Thus by Lemma~\ref{bilinearCRTNint} and H\"older's inequality
\begin{align*}
  \left\|\left( \int_0^1 \left( t^{-\frac{\alpha}{2}}  \e^{-ct\Delta_\chi} F_i(f,g) \, \frac{\dd t}{t} \right)^q \right)^{1/q} \right\|_{L^p}
  &= \left\|\left( \int_0^1 \left( t^{-\frac{\alpha}{2}}  \e^{-ct\Delta_\chi}  | W^{(h)}_t f | |W^{(m+k)}_t g| \, \frac{\dd t}{t} \right)^q \right)^{1/q} \right\|_{L^p}\\
 &\lesssim  \left\|\left( \int_0^1 \left( t^{-\frac{\alpha}{2}}  | W^{(h)}_t f |  |W^{(m+k)}_t g| \, \frac{\dd t}{t} \right)^q \right)^{1/q} \right\|_{L^p}\\
&   \leq  \big\|\sup_{t\in (0,1)} |W^{(m+k)}_t g| \big\|_{L^{p_2}}   \| f\|_{F^{p_1,q}_\alpha}\\
& \lesssim \| g\|_{L^{p_2}} \| f\|_{F^{p_1,q}_\alpha},
\end{align*}
the last inequality since $|W^{(m+k)}_t g| \lesssim e^{-ct\Delta_\chi} |g|$ and by the $L^{p_2}$-boundedness of the local heat maximal function.

Assume now that $i\in \{1,\dots, m-1 \}$. 
Since
\[
 \max_{|L|= i+2h} \max_{|J|= 2m+2k-i} | Y_{L}\e^{-t\Delta_\chi} f \cdot Z_{J} \e^{-t\Delta_\chi}g |
 \leq  \max_{|L|\leq i+2h} | Y_{L}\e^{-t\Delta_\chi} f| \max_{|J|\leq 2m+2k-i}| Z_{J} \e^{-t\Delta_\chi}g |,
\]
one has
\begin{align*}
  &\left\|\left( \int_0^1 \left( t^{-\frac{\alpha}{2}}  \e^{-ct\Delta_\chi} F_i(f,g) \, \frac{\dd t}{t} \right)^q \right)^{1/q} \right\|_{L^p}\\
  & \lesssim \left\|\left( \int_0^1 \left(t^{m+h+k - \frac{\alpha}{2}}  \e^{-ct\Delta_\chi} \left(  \max_{|L|\leq i+2h} | X_{L}\e^{-t\Delta_\chi} f | \max_{|J|\leq 2m+2k-i} |X_{J} \e^{-t\Delta_\chi}g|\right) \, \frac{\dd t}{t} \right)^q \right)^{1/q} \right\|_{L^p}\\
    & \lesssim \left\|\left( \sum_{j=0}^\infty \int_{2^{-j}}^{2^{-j+1}} \left( t^{-\frac{\alpha}{2}}   \e^{-ct\Delta_\chi} \left(W^{(i+2h), *}_{2^{-j}}f\cdot W^{(2m+2k-i), *}_{2^{-j}}g \right) \, \frac{\dd t}{t} \right)^q \right)^{1/q} \right\|_{L^p}\\
     & \lesssim \left\|\left( \sum_{j=0}^\infty \left( 2^{j\frac{\alpha}{2}}   \e^{-c2^{-j}\Delta_\chi} (W^{(i+2h), *}_{2^{-j}}f\cdot W^{(2m+2k-i), *}_{2^{-j}}g ) \right)^q \right)^{1/q} \right\|_{L^p}\\
& \lesssim \left\|\left( \sum_{j=0}^\infty \left( 2^{j \frac{\alpha}{2}} W^{(i+2h), *}_{2^{-j}}f\cdot W^{(2m+2k-i), *}_{2^{-j}}g \right)^q \right)^{1/q} \right\|_{L^p},
\end{align*}
where we used Lemma~\ref{contCRTN}. We apply H\"older's inequality with $\alpha_1 = \frac{i+2h}{2(m+h+k)}\alpha$, $ \alpha_2 =\frac{2m+2k-i}{2(m+h+k)}\alpha$, $\frac{2(m+h+k)}{q_1} = \frac{i+2h}{q}$, $\frac{2(m+h+k)}{q_2} = \frac{2m+2k-i}{q}$ to obtain that the last term of the previous inequality is controlled by
\begin{align*}                            
       &  \Bigg\|\left( \sum_{j=0}^\infty \left( 2^{j\frac{\alpha_1}{2}} W^{(i+2h), *}_{2^{-j}}f \right)^{q_1} \right)^{1/{q_1}}  \left( \sum_{j=0}^\infty \left( 2^{j\frac{\alpha_2}{2}} W^{(2m+2k-i), *}_{2^{-j}}g  \right)^{q_2} \right)^{1/{q_2}}  \Bigg\|_{L^p}\,,  
 \end{align*}      
 which again by H\"older's inequality with  $\frac{2(m+h+k)}{r_1} = \frac{i+2h}{p_1} + \frac{2m+2k-i}{p_3}$, $ \frac{2(m+h+k)}{r_2} = \frac{i+2h}{p_2} + \frac{2m+2h-i}{p_4}$, is in turn controlled by
       \begin{align*}
               &  \left\|\left( \sum_{j=0}^\infty \left( 2^{j\frac{\alpha_1}{2}}W^{(i+2h), *}_{2^{-j}}f \right)^{q_1} \right)^{1/{q_1}}\right\|_{L^{r_1}}  \left\| \left( \sum_{j=0}^\infty \left( 2^{j\frac{\alpha_2}{2}} W^{(2m+2k-i), *}_{2^{-j}}g \right)^{q_2} \right)^{1/{q_2}}  \right\|_{L^{r_2}} \lesssim \|f\|_{F^{r_1,q_1}_{\alpha_1}} \|g\|_{F^{r_2,q_2}_{\alpha_2}},
\end{align*}
the last inequality by Theorem~\ref{teo-equiv3}. Let now $\theta = \frac{i+2h}{2(m+h+k)}$, and observe that by Theorem~\ref{interpolation}
\[
(F^{p_3,\infty}_0(\mu_\chi), F^{p_1,q}_{\alpha}(\mu_\chi))_{[\theta]} = F^{r_1,q_1}_{\alpha_1}(\mu_\chi)
\]
and
\[
(F^{p_4,q}_{\alpha}(\mu_\chi), F^{p_2,\infty}_0(\mu_\chi))_{[\theta]} = F^{r_2,q_2}_{\alpha_2}(\mu_\chi).
\]
Since for every $s\in (1,\infty]$ we have $L^s(\mu_\chi)  \hookrightarrow F^{s,\infty}_0(\mu_\chi)$ by the $L^s$-boundedness of  the heat maximal function, by~\eqref{interpinequality} we have
\begin{align*}
   \|f\|_{F^{r_1,q_1}_{\alpha_1}} \|g\|_{F^{r_2,q_2}_{\alpha_2}}
& \lesssim \|f\|_{L^{p_3}}^\theta \|f\|_{F^{p_1,q}_\alpha}^{1-\theta} \|g\|_{F^{p_4,q}_\alpha}^{\theta}  \|g\|_{L^{p_2}}^{1-\theta} \\
& \lesssim \|f\|_{L^{p_3}} \|g\|_{F^{p_4,q}_\alpha} + \|f\|_{F^{p_1,q}_\alpha}  \|g\|_{L^{p_2}}
\end{align*}
which completes the proof of~\eqref{eqalg2TL} and of the theorem.
\end{proof}

\section{Future developments}\label{geo-inq-sec}  

Although the theory developed in this paper is rather complete, there are additional questions that are certainly worth of investigation. 

In analogy to the Euclidean setting, we expect that the Triebel--Lizorkin spaces $F^{p,2}_0(\mu_\chi)$, when either $p=1$ or $p=\infty$, correspond respectively to the local Hardy space $\mathfrak h^1(\mu_\chi)$ and its dual $\mathfrak{bmo}(\mu_\chi)$ introduced in \cite{BPTV}. We also expect that the dual spaces of $B^{p,q}_{\alpha}(\mu_\chi)$ and $F^{p,q}_{\alpha}(\mu_\chi)$ when $\alpha>0$ can be identified with analogous spaces with negative index of regularity. Moreover, it would be interesting to extend our results to the case when $0<p,q<1$. To conclude, we mention the study of homogeneous versions of Besov and Triebel--Lizorkin spaces on nondoubling Lie groups. 

\bigskip

\emph{Acknowledgements}.  The authors wish to thank Andrea Carbonaro for useful conversations.

\end{document}